\documentclass[a4paper,reqno]{amsart}

\title{Restricted universal groups for right-angled buildings}

\author[J. Bossaert]{Jens Bossaert}

\author[T. De Medts]{Tom De Medts}
\email[Tom De Medts (corresponding author)]{tom.demedts@ugent.be}
\address[Address for both authors]{Ghent University \\ Department of Mathematics: Algebra and Geometry \\ Krijgslaan~281, S25, 9000 Ghent, Belgium}

\date{\today}

\usepackage{amsmath,amssymb,amsthm,mathtools,mathrsfs,graphicx,subcaption,booktabs,numprint,multirow,url}
\usepackage{tikz,tikzpagenodes,tikz-3dplot}
	\usetikzlibrary{calc, decorations, decorations.pathmorphing, decorations.markings, shapes.misc, cd}
	\tikzset{
    	myvertex/.style = {circle, draw, fill, thick, outer sep=3\pgflinewidth, inner sep=0pt, minimum size=3.5pt}, %\bullet-sized
    	mydoublevertex/.style = {myvertex, double, double distance=\pgflinewidth, minimum size=4pt-2\pgflinewidth, outer sep=4\pgflinewidth},
    	myedge/.style = {draw=black, thick, line cap=round, line join=round},
    	every label/.style = {font=\strut, label distance=0pt, outer sep=0pt, inner sep=0pt},
    	baseline = {([yshift=-.57ex]current bounding box.center)},
    	x=12mm, y=12mm,
    }
    
    \pgfdeclaredecoration{cheating dash}{start}{
    	\state{start}[width=\pgfdecoratedremainingdistance, persistent precomputation={
    		\let\on=\pgfdecorationsegmentlength%
    		\let\off=\pgfdecorationsegmentamplitude%
    		\pgfmathparse{\pgfdecoratedremainingdistance-\on}\let\rest=\pgfmathresult%
    		\pgfmathparse{\on+\off}\let\onoff=\pgfmathresult%
    		\pgfmathparse{max(floor(\rest/\onoff), 1)}\let\nfullonoff=\pgfmathresult%
    		\pgfmathparse{max((\rest-\onoff*\nfullonoff)/\nfullonoff+\off, \off)}\global\let\offexpand=\pgfmathresult%
    	}]{\pgfsetpath\pgfdecoratedpath}
    }
    
    \newif\ifstartcompletesineup
    \newif\ifendcompletesineup
    \pgfkeys{
    	/pgf/decoration/.cd,
    	start up/.is if=startcompletesineup,
    	start up=true,
    	start up/.default=true,
    	start down/.style={/pgf/decoration/start up=false},
    	end up/.is if=endcompletesineup,
    	end up=true,
    	end up/.default=true,
    	end down/.style={/pgf/decoration/end up=false}
    }
    \pgfdeclaredecoration{cheating snake}{initial}
    {
    	\state{initial}[
    		width = +0pt,
    		next state = upsine,
    		persistent precomputation = {
    			\ifstartcompletesineup
    				\pgfkeys{/pgf/decoration automaton/next state=upsine}
    				\ifendcompletesineup
    					\pgfmathsetmacro\matchinglength{.5*\pgfdecoratedinputsegmentlength / (ceil(.5*\pgfdecoratedinputsegmentlength / \pgfdecorationsegmentlength))}
    				\else
    					\pgfmathsetmacro\matchinglength{.5*\pgfdecoratedinputsegmentlength / (ceil(.5*\pgfdecoratedinputsegmentlength / \pgfdecorationsegmentlength) - 0.4999)}
    				\fi
    			\else
    				\pgfkeys{/pgf/decoration automaton/next state=downsine}
    				\ifendcompletesineup
    					\pgfmathsetmacro\matchinglength{.5*\pgfdecoratedinputsegmentlength / (ceil(.5*\pgfdecoratedinputsegmentlength / \pgfdecorationsegmentlength) - 0.4999)}
    				\else
    					\pgfmathsetmacro\matchinglength{.5*\pgfdecoratedinputsegmentlength / (ceil(.5*\pgfdecoratedinputsegmentlength / \pgfdecorationsegmentlength) )}
    				\fi
    			\fi
    			\setlength{\pgfdecorationsegmentlength}{\matchinglength pt}
    		}] {}
    	\state{downsine}[width=\pgfdecorationsegmentlength,next state=upsine]{
    		\pgfpathsine{\pgfpoint{.5\pgfdecorationsegmentlength}{.5\pgfdecorationsegmentamplitude}}
    		\pgfpathcosine{\pgfpoint{.5\pgfdecorationsegmentlength}{-.5\pgfdecorationsegmentamplitude}}
    	}
    	\state{upsine}[width=\pgfdecorationsegmentlength,next state=downsine]{
    		\pgfpathsine{\pgfpoint{.5\pgfdecorationsegmentlength}{-.5\pgfdecorationsegmentamplitude}}
    		\pgfpathcosine{\pgfpoint{.5\pgfdecorationsegmentlength}{.5\pgfdecorationsegmentamplitude}}
    	}
    	\state{final}{}
    }
    
    \definecolor{ugentblue}{RGB}{30,100,200}
    \definecolor{ugentyellow}{RGB}{255,210,0}
    \definecolor{ugentsecondary}{RGB}{190,81,144}
    
    \definecolor{ugentred}{RGB}{220,78,40}
    \definecolor{ugentgreen}{RGB}{113,168,96}
    
    \definecolor{mylightgray}{gray}{.94}
    \definecolor{mydarkgray}{gray}{.78}
    
    \tikzset{cheating dash/.code args={on #1 off #2}{
    	\tikzset{decoration={cheating dash, segment length=#1, amplitude=#2}, decorate}%
    	\csname tikz@addoption\endcsname{\pgfsetdash{{#1}{\offexpand}}{0pt}}}
    }
    
    \tikzset{
    	dotted/.style = {very thick, cheating dash = on 2000sp off 3.6pt},
    	dashed/.style = {cheating dash = on 3pt off 3pt},
    	-dots/.style = {postaction = {decoration = {markings, mark=at position 1 with {\node[anchor=180, transform shape, outer sep=2pt, inner sep=0pt] {\kern45000sp\textellipsis};}}, decorate}}
    }
    
    \tikzcdset{every arrow/.append style={line width=.45pt},
    	cells = {nodes = {rounded rectangle, font=\strut, outer sep=2pt, inner sep=2pt}},
    	dims/.style 2 args = {column sep = {#1,between origins}, row sep = {#2,between origins}},
    	dims/.default = {5em}{5em},
    	diagrams = {dims},
    }
    
    \tikzcdset{
    	relay arrow/.default = 10pt,
    	relay arrow/.style = {rounded corners, to path = {
    		\pgfextra{
    			\def\sourcecoordinate{\pgfpointanchor{\tikztostart}{center}}
    			\def\targetcoordinate{\pgfpointanchor{\tikztotarget}{center}}
    			\pgfmathanglebetweenpoints{\sourcecoordinate}{\targetcoordinate}
    			\edef\tempangle{\pgfmathresult}
    			\pgftransformrotate{\tempangle}
    			\pgfmathifthenelse{#1>0}{\tempangle+90}{\tempangle-90}
    			\pgfcoordinate{tempcoord}{\pgfpointanchor{\tikztostart}{\pgfmathresult}}
    		}
    		(tempcoord) -- ([yshift=#1]tempcoord)
    			-- ([yshift=#1]tempcoord-|\tikztotarget.center) \tikztonodes
    			-- (\tikztotarget)
    	}}
    }
    \newcommand\myblock[2][.95\textwidth]{%
    	\bgroup\centering\smallskip%
    	\tikz\node[myedge,draw=black,inner sep=3mm,rounded corners=3mm] {\parbox{#1}{#2}};%
    	\smallskip\par\egroup}

\usepackage[textsize=footnotesize,textwidth=20ex,colorinlistoftodos]{todonotes}

\usepackage{enumitem}
	\setenumerate[0]{label={\rm (\roman*)},leftmargin=*,itemsep=.2ex}
\usepackage[symbol]{footmisc}
\usepackage{stackengine}
\usepackage{hyperref}
	\hypersetup{bookmarksopen=true}
\usepackage[capitalize]{cleveref}
    \newcommand{\itemref}[2]{\cref{#1}\ref{#1:#2}}

\renewcommand\footnotemark{}

\newtheorem{theorem}{Theorem}[section]
\newtheorem*{theorem*}{Theorem}
\newtheorem{lemma}[theorem]{Lemma}
\newtheorem{proposition}[theorem]{Proposition}
\newtheorem{corollary}[theorem]{Corollary}

\newtheoremstyle{step}{\topsep}{\topsep}{\itshape}{0pt}{\itshape}{\@. }{0pt}{\thmname{#1}\thmnumber{ #2}\thmnote{ (#3)}}
\theoremstyle{step}

\theoremstyle{definition}
\newtheorem{definition}[theorem]{Definition}
\newtheorem{remark}[theorem]{Remark}

\DeclareMathOperator{\Sym}{Sym}
\DeclareMathOperator{\Aut}{Aut}

\DeclareMathOperator{\U}{\mkern2mu\mathcal{U}}
\DeclareMathOperator{\G}{\mkern1mu\mathcal{G}}

\makeatletter
\newcommand{\@Fhack}{\@ifnextchar,{\mkern-2mu}{}}
\newcommand{\F}{\ensuremath\boldsymbol{F}\@Fhack}
\newcommand{\bL}{\ensuremath\boldsymbol{L}\@Fhack}
\newcommand{\Facute}{\ensuremath\boldsymbol{\smash{\acute F}}\@Fhack}
\newcommand{\Fwidehat}{\ensuremath\boldsymbol{\widehat F}\@Fhack}

\renewcommand{\P}{\ensuremath\mathcal{P}}
\renewcommand{\H}{\ensuremath\mathscr{H}}
\newcommand{\R}{\ensuremath\mathcal{R}}
\newcommand{\T}{\ensuremath\mathcal{T}}

\newcommand{\ball}{\ensuremath\operatorname{\mathsf{B}}}
\newcommand{\sphere}{\ensuremath\operatorname{\mathsf{S}}}

\DeclareMathOperator{\Res}{Res}
\DeclareMathOperator{\proj}{proj}
\DeclareMathOperator{\dist}{dist}

\newcommand{\subgroupindex}[2]{\bigl[ #1 \mathrel{:} #2\bigr]}

\newcommand\restrict[3][]{\mathchoice{{#2\bigr\rvert_{#3}^{#1}}}{{#2\rvert_{#3}^{#1}}}{{#2\rvert_{#3}^{#1}}}{{#2\rvert_{#3}^{#1}}}}

\newcommand\transversal[1][\@nil]{\def\tmp{#1}\ifx\tmp\@nnil\Upsilon\else\Upsilon_{\!#1}\fi}
\makeatother

\newcommand{\dircup}{\mathop{\overrightarrow{\bigcup}}\displaylimits}

\newcommand\acts{\mkern 2mu\relax.\mkern 2mu\relax}

\makeatletter % amsart does not yet support MSC2020
\@namedef{subjclassname@2020}{\textup{2020} Mathematics Subject Classification}
\makeatother


\begin{document}

\begin{abstract}
    In 2000, Marc Burger and Shahar Mozes introduced universal groups acting on trees. Such groups provide interesting examples of totally disconnected locally compact groups. Intuitively, these are the largest groups for which all local actions satisfy a prescribed behavior.
    
    Since then, their study has evolved in various directions. In particular, Adrien Le Boudec has studied \emph{restricted} universal groups, where the prescribed behavior is allowed to be violated in a finite number of vertices. On the other hand, we have been studying universal groups acting on \emph{right-angled buildings}, a class of geometric objects with a much more general structure than trees.
    
    The aim of the current paper is to combine both ideas: we will study restricted universal groups acting on right-angled buildings. We show several permutational and topological properties of those groups, with as main result a precise criterion for when these groups are simple.
\end{abstract}

\maketitle

%------------------------------------------------------------------------

\section{Introduction}

In their seminal paper \cite{burgermozes2000}, Marc Burger and Shahar Mozes have introduced \emph{universal groups} acting on trees. More precisely, given any permutation group $F$ acting on a set $S$, there is a ``largest'' group $\U(F)$ acting on the regular tree of degree $|S|$ such that all local actions (i.e., the induced actions on the stars around a vertex and its image) belong to $F$.

If $S$ is finite (an assumption which is often, but certainly not always, imposed), then these universal groups are \emph{totally disconnected locally compact} (tdlc, for short) with respect to the permutation topology. They provide an interesting class of examples, and the study of and the interplay between their permutational and topological properties leads to a better understanding of tdlc groups; see, for instance, \cite{CDM11} and the more recent survey \cite{universal-survey}.

In 2016, Adrien Le Boudec \cite{leboudec} has generalized the notion to so-called \emph{restricted} universal groups---the terminology is due to \cite{caprace_dense}---by allowing for a finite number of \emph{singularities}, i.e., vertices where the local action does not belong to $F$. This gives rise to non-closed subgroups of the automorphism group of the tree, but with the correct topology, these groups are again tdlc groups.

In a different direction, we have been studying universal groups for \emph{right-angled buildings}. These geometric objects can be seen as higher-dimensional analogues of trees, but their combinatorics are governed by a (right-angled) Coxeter diagram, allowing for many different possible structures.
The starting point is our joint work with Ana C. Silva and Koen Struyve in \cite{silva1}, which has been continued in \cite{silva2,bossaert,cityproducts}.
This builds upon earlier results (for the full automorphism group) by Pierre-Emmanuel Caprace in~\cite{caprace2014}; in particular, he had shown that the full automorphism group of a right-angled building is always simple.

In the current paper, we apply Le Boudec's idea on the setting of locally finite right-angled buildings. We introduce restricted universal groups $\G(\F,\Facute)$ acting on a right-angled building $\Delta$ (of rank $n \geq 2$) with respect to ``local data'' $\F$ and $\Facute$, each of which consists of $n$ permutation groups. Again, with the correct topology, this gives rise to new examples of tdlc groups.

The combinatorics of the Coxeter diagram enter the picture in interesting ways. An important aspect is the occurrence of \emph{ladders} and their associated \emph{rungs} in the diagram, which are nodes for which the generators commute with (at least) two non-commuting generators; see \cref{def:ladder} below. The main result of our paper (\cref{thm:restrictedvirtsimple}), which gaves a good flavor for the kind of conditions that arise, is the following.

\begin{theorem*}
	Let $\Delta$ be a thick irreducible right-angled building over an index set $I$ with $\lvert I \rvert \geq 2$. Let $\F$ and $\Facute$ be the local data as in \cref{def:restricted2}. Assume that $F_i = \acute F_i$ for every $i\in I$ that is the type of a rung. Moreover, assume that not all local groups $\acute F_i$ are free.%TODO NODIG?
	
	Then the restricted universal group $\G(\F,\Facute)$ is virtually simple if and only if $\acute F_i$ is generated by point stabilizers for every $i\in I$ and transitive for every $i$ in some vertex cover of the diagram of $\Delta$.
\end{theorem*}

Along the way, we show several other permutational and topological properties for these restricted universal groups.

We should point out that this is (hopefully) not the end of the story. Indeed, in \cite{leboudec}, these restricted universal groups for trees are studied more deeply, in particular with regard to their asymptotic dimension and the (non-)existence of lattices. It is not unlikely that our new class of tdlc groups admits a similar behavior as the groups introduced and studied by Le Boudec.

\subsection*{Acknowledgment.}

The first author has been supported by the UGent BOF PhD mandate BOF17/DOC/274.

\subsection*{Outline of the paper.}

In \cref{sec:pre}, we recall the required background for universal groups for right-angled buildings from \cite{silva1,bossaert,cityproducts}. This includes \emph{colorings}, \emph{universal groups}, \emph{parallel} residues, and the \emph{permutation topology} on right-angled buildings.

We then need a fairly long \cref{sec:panel-closed} to show an extension result starting from a \emph{panel-closed} subset of a right-angled building. This requires generalizing some earlier results from \cite{silva1} (such as the \emph{Closing Squares Lemmas} and the existence of \emph{concave galleries}) to a more general setting. Our main result in that section is \cref{prop:evenmoretechnicalextensionstuff}.

We then proceed to the \emph{restricted universal groups}, which form the main subject of this paper. In \cref{sec:restrunivdef}, we first define the groups $\G(\F)$ in \cref{def:restricted1} and point out the relevance of the existence of \emph{ladders} (and their \emph{rungs}) in the building, and we then proceed to the more general groups $\G(\F, \Facute)$ in \cref{def:restricted2}.

We continue in \cref{sec:permutational} to study some of their permutational properties. Using the results from \cref{sec:panel-closed}, we can produce an interesting generating set for the restricted universal groups; this is the content of \cref{prop:restrictedcompgen}.

In \cref{sec:topological}, we then come to the topological structure on $\G(\F, \Facute)$, making it into a tdlc group that is dense inside the usual universal group $\U(\Facute)$ (under some natural assumptions on the local data). Relying on the results from \cref{sec:permutational}, we show that the groups are always compactly generated; this is \cref{cor:restrictedcompgen}.

In the final \cref{sec:simplicity}, we then study the simplicity question. This first requires a careful analysis of when an action of a group on a right-angled building is \emph{combinatorially dense}. We then proceed to show that the groups $\G(\F,\Facute)$ are almost always \emph{monolithic} and have a \emph{simple monolith} (\cref{prop:restricteduniversalmonolithic}). This then leads us to our final simplicity result (\cref{thm:restrictedvirtsimple}) that we already mentioned above.

%--------------------------------------------------------

\section{preliminaries: Universal groups for right-angled buildings}\label{sec:pre}

We will adopt the conventions and notations from \cite{cityproducts}. In particular, $I$~will always be the index set of a Coxeter system $M$, which we will assume to be \emph{right-angled} throughout the paper, i.e., all entries in the Coxeter matrix are equal to $1$, $2$ or $\infty$.

To introduce universal groups for right-angled buildings, we follow the approach from \cite{silva1}.
Our starting point is the well known fact, attributed to Fr\'ed\'eric Haglund and Fr\'ed\'eric Paulin \cite{haglundpaulin}, that for each choice of parameters $(q_i)_{i \in I}$, there exists a unique semiregular right-angled building $\Delta$ of type $M$ with these parameters (i.e., each $i$-panel contains precisely $q_i$ chambers).
In addition, such a building $\Delta$ admits an (essentially unique) \emph{legal coloring}.

\begin{definition}[\cite{silva1}]
	\label{def:coloring}
	Let $\Delta$ be a semiregular right-angled building of type~$M$ over $I$ with parameters $(q_i)_{i \in I}$.
	For each $i\in I$, consider a set $\Omega_i$ of cardinality~$q_i$, the elements of which we call \emph{$i$-colors} or \emph{$i$-labels}.
	A \emph{legal coloring} of~$\Delta$ (with color sets $\Omega_i$) is a map
	\[ \lambda\colon \Delta\to\prod_{i\in I} \Omega_i\colon c\mapsto(\lambda_i(c))_{i\in I} \]
	satisfying the following properties for every $i\in I$ and for every $i$-panel $\P$:
	\begin{enumerate}
		\item the restriction $\restrict{\lambda_i}{\P}\colon \P\to\Omega_i$ is a bijection;
		\item for every $j\neq i$, the restriction $\restrict{\lambda_j}{\P}\colon \P\to\Omega_j$ is a constant map.
	\end{enumerate}
	%In other words, for every pair of adjacent chambers $c\sim_i d$, we require that $\lambda_i(c)\neq \lambda_i(d)$ but that $\lambda_j(c) = \lambda_j(d)$ for all $j\neq i$.
    By \cite[Proposition 2.44]{silva1}, such a legal coloring is essentially unique.
\end{definition}

\myblock{%
    From now on, we will always assume that
    \begin{itemize}\setlength{\itemindent}{-3ex}
        \item $M$ is a right-angled Coxeter system over an index set $I$,
        \item $\Delta$ is a semiregular right-angled building of type $M$ with parameters $(q_i)_{i\in I}$,
        \item each $\Omega_i$ is a set of cardinality $q_i$,
        \item $\lambda$ is a legal coloring of $\Delta$ with color sets $\Omega_i$.
    \end{itemize}
}

\begin{definition}
	\label{def:universal}
    \begin{enumerate}
        \item 
            Consider an automorphism $g\in\Aut(\Delta)$ and an arbitrary $i$-panel $\P$. Then we define the \emph{local action of $g$ at $\P$} as the map
        	\[\sigma_\lambda(g,\P) = \restrict{\lambda_i}{g\P} \circ \restrict{g}{\P} \circ \restrict[-1]{\lambda_i}{\P}, \]
        	which is a permutation of $\Omega_i$ by definition of $\lambda$.
        	In other words, the local action $\sigma_\lambda(g,\P)$ is the map that makes the following diagram commute.
        	\[\begin{tikzcd}[dims={7em}{4em}]
        		\P \ar[r,"g"] \ar[d,"\lambda_i"]
        			& g\P \ar[d,"\lambda_i"]\\
        		\Omega_i \ar[r,"{\sigma_\lambda(g,\,\P)}"] & \Omega_i
        	\end{tikzcd}\]
        \item 
        	Let $\F$ be a collection of permutation groups $F_i\leq\Sym(\Omega_i)$, indexed by $i\in I$. The \emph{universal group} of $\F$ over $\Delta$ (with respect to $\lambda$) is the group
        	\[
        	   \U(\F) = \bigl\{g\in\Aut(\Delta) \bigm| \sigma_\lambda(g,\P)\in F_i  \text{ for each $i\in I$ and each $\P\in\Res_i(\Delta)$}\bigr\},
        	\]
        	i.e., $\U(\F)$ is the group of automorphisms that locally act like permutations in $F_i$. We call the groups $F_i$ the \emph{local groups} and we refer to the collection $\F$ as the \emph{local data} for $\U(\F)$.
        	% Note that $\U_\Delta^\lambda(\F)$ is indeed a subgroup of $\Aut(\Delta)$ by \cref{lem:localcomposition}.
        	We will sometimes explicitly include the building $\Delta$ in the notation and write $\U_\Delta(\F)$ instead.
    \end{enumerate}
\end{definition}

\begin{lemma}
	\label{lem:localcomposition}
	Let $g,h\in\Aut(\Delta)$ and let $\P$ be any panel. Then the local actions satisfy
	\[\sigma_\lambda(gh,\P) = \sigma_\lambda(g,h\P)\cdot\sigma_\lambda(h,\P)
		\qquad\text{and}\qquad
		\sigma_\lambda(g,\P)^{-1} = \sigma_\lambda(g^{-1},g\P).\]
\end{lemma}
\begin{proof}
	This follows immediately from the definition.
\end{proof}

The following extension result is useful.
\begin{proposition}\label{prop:ext-res}
    Let $\R$ be a residue of $\Delta$ of type $J \subseteq I$. Let $g \in \U_\R(\restrict{\F}{J})$. Then $g$ extends to an element of $\U_\Delta(\F)$.
\end{proposition}
\begin{proof}
    This is \cite[Proposition 3.15]{silva1}.
\end{proof}

The following definition is taken from \cite[Definition 2.17]{bossaert}.
\begin{definition}\label{def:harmonious}
    Let $\F$ be local data over $I$ and let $J \subseteq I$.
    Two $J$-residues $\R$ and $\R'$ are called \emph{harmonious} (with respect to $\F$) if for each $k \in I \setminus J$, the (well-defined) colors $\lambda_k(\R)$ and $\lambda_k(\R')$ lie in the same orbit of the local group $F_k$.
\end{definition}
We will use this concept mostly when $\R$ and $\R'$ are chambers (so $J = \emptyset$) or when $\R$ and $\R'$ are panels (so $J = \{ j \}$ is a singleton).
\begin{proposition}\label{prop:universalorbits}
    Two residues $\R$ and $\R'$ lie in the same orbit of $\U(\F)$ if and only if they are harmonious.
\end{proposition}
\begin{proof}
    This is \cite[Proposition 3.1]{bossaert}.
\end{proof}

\medskip

We will now have a closer look at local actions on \emph{parallel residues}, a notion introduced by Pierre-Emmanuel Caprace in \cite{caprace2014}.

\begin{definition}
\begin{enumerate}
    \item 
        If $\R_1$ and $\R_2$ are residues of $\Delta$, we set
        \[\proj_{\R_2}(\R_1) = \{\proj_{\R_2}(c)\mid c\in\R_1\}.\]
        Note that this set is again a residue of $\Delta$ contained in $\R_2$, and that its rank is bounded by the ranks of both $\R_1$ and $\R_2$.
    \item 
    	Two residues $\R_1$ and $\R_2$ of $\Delta$ are called \emph{parallel} if $\proj_{\R_1}(\R_2) = \R_1$ and $\proj_{\R_2}(\R_1) = \R_2$. In that case, the projection maps define bijections between $\R_1$ and $\R_2$.
    \item 
    	Let $J\subseteq I$. Then we define the set
    	\[J^\perp = \{i\in I\mid \text{$m_{ij}=2$ for all $j\in J$}\} = \{i\in I\setminus J \mid \text{$ij=ji$ for all $j\in J$}\}.\]
    	When $J=\{j\}$ is a singleton, we simply write $j^\perp := \{j\}^\perp$.
\end{enumerate}
\end{definition}

\begin{proposition}
	\label{prop:parallelinrab}
	Let $\Delta$ be a right-angled building over $I$.
	\begin{enumerate}
		\item\label{prop:parallelinrab:1} Two parallel residues have equal type.
		\item\label{prop:parallelinrab:2} Two residues of equal type $J$ are parallel if and only if they are contained in a common residue of type $J\cup J^\perp$.
		\item\label{prop:parallelinrab:3} Parallelism of residues is an equivalence relation.
		\item\label{prop:parallelinrab:4} Two panels $\P_1$ and $\P_2$ are parallel if and only if there exist two chambers in $\P_1$ with distinct projections on $\P_2$.
	\end{enumerate}
\end{proposition}
\begin{proof}
	We refer to \cite[Lemma 2.5, Proposition 2.8 and Corollary 2.9]{caprace2014}.
\end{proof}

% NO LONGER NEEDED!
%
%\begin{lemma}\label{lem:homotopy}
%    Let $c,d \in \Delta$ and let $\gamma_1,\gamma_2$ be two minimal galleries from $c$ to $d$. Then a homotopy transforming the type of $\gamma_1$ into the type of $\gamma_2$ transforms each panel on $\gamma_1$ to a parallel panel on $\gamma_2$.
%\end{lemma}
%\begin{proof}
%    It suffices to show this for an \emph{elementary} homotopy $w \; ij \; w' \mapsto w \; ji \; w'$ where $m_{ij}=2$, and in this case, it is obvious that the $i$-panels and $j$-panels are transformed to parallel panels (e.g., using \itemref{prop:parallelinrab}{2}).
%\end{proof}

\begin{proposition}
	\label{prop:parallellocals}
	Let $g \in \Aut(\Delta)$. If\, $\P$ and $\P'$ are two parallel panels in $\Delta$, then the local actions $\sigma_\lambda(g,\P)$ and $\sigma_\lambda(g,\P')$ are identical.
\end{proposition}
\begin{proof}
    This is \cite[Lemma 3.7]{bossaert}.
\end{proof}

\medskip

Automorphism groups of right-angled buildings come equipped with a natural topology.
\begin{definition}
    We define the \emph{permutation topology} on $\Aut(\Delta)$ by taking as an identity neighborhood basis the collection of all pointwise stabilizers of finite subsets of $\Delta$. This makes $\Aut(\Delta)$ (and all its subgroups) into a totally disconnected group.
    
    If all parameters $q_i$ are finite---this will always be the case later on---then $\Aut(\Delta)$ is also locally compact, and the same is then true for the universal groups $\U(\F)$.
\end{definition}

The following notion will play an important role later.
\begin{definition}
	\label{def:young}
	\begin{enumerate}
	    \item 
        	To every partition of a set $X$, we can associate a subgroup of $\Sym(X)$ of all permutations stabilizing the blocks of the partition. A subgroup obtained in this fashion is called a \emph{Young subgroup} of $\Sym(X)$, and is naturally isomorphic to the direct product of the symmetric groups on the blocks.
	    \item 
        	For any permutation group $G\leq\Sym(X)$, we have a canonical partition of $X$ into $G$-orbits. We call the Young subgroup associated to this partition the \emph{Young overgroup of $G$} and denote it by $\widehat G$.
	\end{enumerate}
\end{definition}

Note that we indeed always have the inclusions $G \leq \widehat G \leq \Sym(\Omega)$. Moreover, $G = \widehat G$ if and only if $G$ is itself a Young subgroup, and $\widehat G = \Sym(\Omega)$ if and only if $G$ is transitive.

%------------------------------------------------------------------------

\section{Panel-closed subsets of right-angled buildings}\label{sec:panel-closed}

Our aim in this section is to prove an extension result, which is somewhat similar to \cite[Proposition 3.8]{bossaert} but requires a more subtle setup. For this purpose, we introduce the notion of panel-closed subsets of right-angled buildings.

\begin{definition}
    A subset $C$ of $\Delta$ is called \emph{panel-closed} if:
    \begin{itemize}
        \item $C$ is convex (i.e., if $c,d \in C$, then all chambers on all minimal galleries between $c$ and $d$ are also contained in $C$), and
        \item each panel $\P$ of $\Delta$ containing at least $2$ chambers of $C$ is completely contained in $C$.
    \end{itemize}
\end{definition}

Our first goal will be to generalize the ``closing squares lemmas'' and the existence of ``concave galleries'' from \cite{silva1}.
Let us first recall these facts.
\begin{lemma}[Closing squares]\label{lem:closingsquares}
    \begin{enumerate}
        \item \label{lem:closingsquares:1}
            Let $c, c_1, c_2, c_3 \in \Delta$ be such that $\dist(c,c_1) = \dist(c,c_3) = n$, $\dist(c,c_2) = n+1$, and $c_1\sim_i c_2\sim_j c_3$ for $i\neq j$. Then $m_{ij}=2$ and there exists $d\in\Delta$ such that $\dist(c,d) = n-1$ and $c_1\sim_j d\sim_i c_3$.
        \item \label{lem:closingsquares:2}
            Let $c, c_1, c_2, c_3 \in \Delta$ be such that $\dist(c,c_1) = \dist(c,c_2) = n+1$, $\dist(c,c_3) = n$, and $c_1\sim_i c_2\sim_j c_3$ for $i\neq j$. Then $m_{ij}=2$ and there exists $d\in\Delta$ such that $\dist(c,d) = n$ and $c_1\sim_j d\sim_i c_3$.
    \end{enumerate}
\end{lemma}
\begin{proof}
    See \cite[Lemmas 2.9 and 2.10]{silva1}; see also \cref{fig:closingsquares} below.
\end{proof}
\begin{figure}[ht]
	\centering
	\begin{subfigure}[b]{.4\textwidth}
		\centering
		\begin{tikzpicture}
			\node[myvertex,label=330:$c$] (C0) at (0,1.5) {};
			\foreach\i/\l in {0/$n-1$,1/$n$,1.9/$n+1$}
				\draw[myedge,mydarkgray] (65:3+\i) node[anchor=170] {\strut\l} arc (65:115:3+\i);
			%nodes
			\node[myvertex,label=330:$d$] (D) at (90:3) {};
			\node[myvertex,label=150:$c_1$] (C1) at (100:4) {};
			\node[myvertex,label=30:$c_3$] (C3) at (80:4) {};
			\node[myvertex,label=30:$c_2$] (C2) at (90:4.9) {};
			%paths
			\draw[myedge,decoration={start down,end up,amplitude=2mm,cheating snake},decorate] (C0.120) to (D.240);
			\draw[myedge,ugentblue] (C1) to (C2) (C3) to (D);
			\draw[myedge,ugentred] (C2) to (C3) (C1) to (D);
		\end{tikzpicture}
		\caption{\itemref{lem:closingsquares}{1}}
	\end{subfigure}\hspace*{5ex}
	\begin{subfigure}[b]{.4\textwidth}
		\centering
		\begin{tikzpicture}[every label/.append style={black}]
			\node[myvertex,label=330:$c$] (C0) at (0,2) {};
			\draw[myedge,mydarkgray] (65:4) node[anchor=170] {\strut $n$}
				arc (65:80:4) node[myvertex,black,label=10:$c_3$] (C3) {}
				(100:4) node[myvertex,black,label=170:$d$] (D) {}
				arc (100:115:4);
			\draw[myedge,mydarkgray] (65:4.9) node[anchor=170] {\strut $n+1$}
				arc (65:81:4.9) node[myvertex,black,label=30:$c_2$] (C2) {}
				(99:4.9) node[myvertex,black,label=150:$c_1$] (C1) {}
				arc (99:115:4.9);
			\draw[myedge,decoration={start down,end down,amplitude=2mm,cheating snake},decorate] (C0.150) to (D.300);
			\draw[myedge,decoration={start up,end up,amplitude=2mm,cheating snake},decorate] (C0.30) to (C3.240);
			\draw[myedge,ugentblue] (C1) to (C2) (C3) to (D);
			\draw[myedge,ugentred] (C1) to (D) (C2) to (C3);
		\end{tikzpicture}
		\caption{\itemref{lem:closingsquares}{2}}
	\end{subfigure}
\caption{The ``closing squares'' lemmas}
\label{fig:closingsquares}
\end{figure}
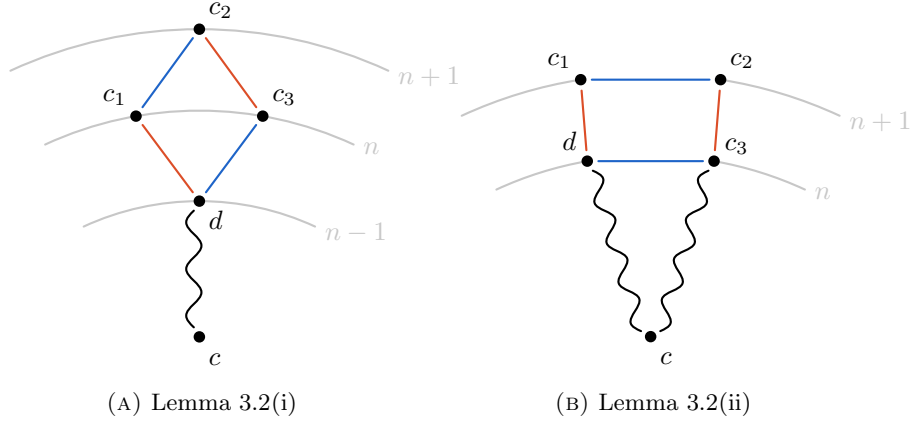

%We will also need the notion of \emph{concave galleries} from \cite[Lemma 2.11]{silva1}.
\begin{proposition}[Concave galleries]\label{prop:concave}
    Let $c, c_1, c_2 \in \Delta$. Then there exists a minimal gallery
    \[ \gamma = (v_0,\dots,v_\ell), \quad \text{with } v_0 = c_1,\ v_\ell = c_2, \]
    such that there are numbers $0 \leq j \leq k \leq \ell$ satisfying:
    \begin{itemize}
        \item $\dist(c, v_i) < \dist(c, v_{i-1})$ for all $i \in \{ 1,\dots, j \}$,
        \item $\dist(c, v_i) = \dist(c, v_{i-1})$ for all $i \in \{ j+1,\dots, k \}$,
        \item $\dist(c, v_i) > \dist(c, v_{i-1})$ for all $i \in \{ k+1,\dots, \ell \}$.
    \end{itemize}
    Such a minimal gallery from $c_1$ to $c_2$ will be called \emph{concave (with respect to $c$)}.
\end{proposition}
\begin{proof}
    This is \cite[Lemma 2.11]{silva1}.
\end{proof}

The following lemma is a generalization of the gate property for residues (see, e.g., \cite[Proposition 5.34]{abramenkobrown}); note that residues are indeed panel-closed.

\begin{lemma}
	\label{lem:convexgateproperty}
	Let $C$ be a panel-closed subset of $\Delta$ and let $c\in\Delta$ be any chamber. Then there is a unique chamber $d\in C$ such that $\dist(c,d)$ is minimal.
	Moreover, the gate property holds: for every chamber $e\in C$, there exists a minimal gallery from $c$ to $e$ passing through $d$.
\end{lemma}
\begin{proof}
	Assume that $d$ and $d'$ are two distinct chambers in $C$ that minimize the distance to $c$. Let $\gamma$ be the minimal gallery that joins $d$ and $d'$ and is concave with respect to $c$ (as in \cref{prop:concave}). By convexity of~$C$ and concavity of~$\gamma$, every chamber on $\gamma$ is a chamber in $C$ at the same distance to $c$ as $d$ and $d'$. We may thus assume that $d$ and $d'$ are $i$-adjacent for some $i\in I$. Let $\P$ be the $i$-panel containing $d$ and $d'$; by assumption, $\P \subseteq C$. Now let $c'=\proj_\P(c)$; then $\dist(c,c')<\dist(c,d)$, while $c'\in C$. This is a contradiction, showing that there is a unique chamber $d\in C$ closest to $c$.
	
	For the gate property, we use induction on $\dist(c,C)$. If $c\in C$, then there is nothing to show. Now assume $\dist(c,C)\geq 1$ and let $c'\in\Delta$ and $i\in I$ be such that $\dist(c',C)<\dist(c,C)$ and $c'\sim_i c$. Let $\P=\P_i(c)$. Then by construction, $\proj_\P(d)=c'$ and there exists a minimal gallery $\gamma$ from $c'$ to $e$ through $d$ by the induction hypothesis. Minimality yields that also $\proj_\P(e)=c'$. Hence $\dist(c,e)>\dist(c',e)$ and the gallery obtained by prepending $c\sim_i c'$ to $\gamma$ is again minimal.
\end{proof}
\begin{definition}
    Let $C$ be a panel-closed subset of $\Delta$ and let $c\in\Delta$.
    We will call the unique chamber $d \in C$ closest to $c$ the \emph{projection} of $c$ on $d$ and we will denote it by $\proj_C(c)$.
\end{definition}

% To proceed, we will use and generalize an auxiliary result from \cite{silva1}.
\begin{definition}
\begin{enumerate}
    \item
        Let $c \in \Delta$ be a chamber. As in \cite[Definition 1.3]{silva1}, we define the \emph{ball} and the \emph{sphere} of radius $n$ around $c$ by
    	\begin{align*}
    		\ball_n(c) &= \{d\in\Delta \mid \dist(c,d)\leq n\},\\
    		\sphere_n(c) &= \{d\in\Delta \mid \dist(c,d)= n\}.
    	\end{align*}
    \item
        More generally, let $C \subseteq \Delta$ be panel-closed. Then we define the \emph{ball} and the \emph{sphere} of radius $n$ around $C$ by
    	\begin{align*}
    		\ball_n(C) &= \{d\in\Delta \mid \dist(C,d)\leq n\},\\
    		\sphere_n(C) &= \{d\in\Delta \mid \dist(C,d)= n\}.
    	\end{align*}
\end{enumerate}
\end{definition}

We are now ready to generalize \cref{lem:closingsquares} and \cref{prop:concave} to panel-closed subsets.
We begin with an easy lemma, emphasizing an important dichotomy.
\begin{lemma}\label{lem:C-spheres}
	\label{lem:panel-to-C}
	Let $C \subseteq \Delta$ be panel-closed and let $\P$ be an $i$-panel.
	Let $n$ be the smallest number such that $\P \cap \sphere_n(C) \neq \emptyset$.
	Then either
	\begin{enumerate}[label={\rm (\alph*)}]
	    \item\label{lem:C-spheres:a} \begin{itemize}
	        \item $\P \subseteq \sphere_n(C)$,
	        \item $\P$ is parallel to a panel $\P' \subseteq C$, and
	        \item $\proj_C(\P) = \P'$; \quad or
	    \end{itemize} \vspace*{1ex}
	    \item\label{lem:C-spheres:b} \begin{itemize}
	        \item $\P$ has exactly one chamber in $\sphere_n(C)$ and all its other chambers in $\sphere_{n+1}(C)$,
	        \item $\P$ is not parallel to any panel in $C$, and
	        \item $\proj_C(\P) = \{ e \}$ for some $e \in C$.
	    \end{itemize}
	\end{enumerate}
\end{lemma}
\begin{proof}
    Let $c \in \P \cap \sphere_n(C)$. Since every other chamber in $\P$ is adjacent to $c$, we already have $\P \subseteq \sphere_n(C) \cup \sphere_{n+1}(C)$.
    
	Assume that there is a second chamber $d \in \P\cap\sphere_n(C)$ with $d \neq c$.  If $\proj_C(c)=\proj_C(d) =: e$, then $\proj_\P(e)$ will be a chamber in $\sphere_{n-1}(C)$ (because of the gate property applied to $\P$), contradicting the minimality of $n$.
	Hence we must have $\proj_C(c)\neq\proj_C(d)$. Those projections are $i$-adjacent; let $\P'$ be the $i$-panel through them. By \itemref{prop:parallelinrab}{4}, $\P$ and $\P'$ are parallel, and since $C$ is panel-closed, we have $\P' \subseteq C$.
	In particular, all chambers in $\P$ are at the same distance from $C$, so $\P\subseteq\sphere_n(C)$.
	
	Assume now that $c$ is the only chamber in $\P \cap \sphere_n(C)$, so that all its other chambers lie in $\sphere_{n+1}(C)$.
	Since not all chambers of $\P$ lie at the same distance from $C$, the panel $\P$ can certainly not be parallel to any panel in $C$.
	In particular, \itemref{prop:parallelinrab}{4} now implies that all chambers in $\P$ must have the same projection on~$C$.
\end{proof}

\begin{lemma}[Closing squares]\label{lem:C-closingsquares}
    Let $C \subseteq \Delta$ be panel-closed.
    \begin{enumerate}
        \item \label{lem:C-closingsquares:1}
            Let $c_1, c_2, c_3 \in \Delta$ be such that $\dist(C,c_1) = \dist(C,c_3) = n$, $\dist(C,c_2) = n+1$, and $c_1\sim_i c_2\sim_j c_3$ for $i\neq j$. Then $m_{ij}=2$ and there exists $d\in\Delta$ such that $\dist(C,d) = n-1$ and $c_1\sim_j d\sim_i c_3$.
        \item \label{lem:C-closingsquares:2}
            Let $c_1, c_2, c_3 \in \Delta$ be such that $\dist(C,c_1) = \dist(C,c_2) = n+1$, $\dist(C,c_3) = n$, and $c_1\sim_i c_2\sim_j c_3$ for $i\neq j$. Then $m_{ij}=2$ and there exists $d\in\Delta$ such that $\dist(C,d) = n$ and $c_1\sim_j d\sim_i c_3$.
    \end{enumerate}
\end{lemma}
\begin{proof}
    \begin{enumerate}
        \item
            By \cref{lem:C-spheres}, the chambers $c_1$, $c_2$ and $c_3$ all have the same projection on~$C$.
            We can thus immediately apply \itemref{lem:closingsquares}{1}.
        \item
            By \cref{lem:C-spheres}, the chambers $c_2$ and $c_3$ have the same projection $c$ on~$C$.
            Let $d := \proj_C(c_1)$.
            If also $d = c$, then we can apply \itemref{lem:closingsquares}{2}.
            If $d \neq c$, then the $i$-panel $\P$ through $c_1$ and $c_2$ is parallel to the $i$-panel $\P'$ through $c$ and $d$, which is contained in $C$ because $C$ is panel-closed.
            By \itemref{prop:parallelinrab}{2}, $\P$ and $\P'$ are contained in a common residue of type $i \cup i^\perp$. In particular, the minimal gallery from $c_2$ to $c$ through $c_3$ is of this type, hence $j \in i^\perp$.
            
            Since $ij = ji$, there exists a chamber $d \in \Delta$ with $c_1 \sim_j d \sim_i c_3$.
            Let $\P''$ be the $i$-panel containing $d$ and $c_3$. By construction, $\P''$ is parallel to~$\P$, but since parallelism is an equivalence relation (\itemref{prop:parallelinrab}{3}), this implies that $\P''$ and $\P' \subseteq C$ are also parallel. Since $\dist(C,c_3)=n$, this implies that also $\dist(C,d)=n$, and we have found the required configuration.
        \qedhere
    \end{enumerate}
\end{proof}
\begin{proposition}[Concave galleries]\label{prop:C-concave}
    Let $C \subseteq \Delta$ be panel-closed and let $c_1, c_2 \in \Delta$. Then there exists a minimal gallery
    \[ \gamma = (v_0,\dots,v_\ell), \quad \text{with } v_0 = c_1,\ v_\ell = c_2, \]
    such that there are numbers $0 \leq j \leq k \leq \ell$ satisfying:
    \begin{itemize}
        \item $\dist(C, v_i) < \dist(C, v_{i-1})$ for all $i \in \{ 1,\dots, j \}$,
        \item $\dist(C, v_i) = \dist(C, v_{i-1})$ for all $i \in \{ j+1,\dots, k \}$,
        \item $\dist(C, v_i) > \dist(C, v_{i-1})$ for all $i \in \{ k+1,\dots, \ell \}$.
    \end{itemize}
    Such a minimal gallery from $c_1$ to $c_2$ will be called \emph{concave (with respect to $C$)}.
\end{proposition}
\begin{proof}
    The proof from \cite[Lemma 2.11]{silva1} for the case $C = \{ c \}$ only relies on the closing squares lemmas. By \cref{lem:C-closingsquares}, the proof carries over \emph{mutatis mutandis} to the case where $C$ is any panel-closed subset of $\Delta$.
\end{proof}

\medskip

Our next goal is to generalize the following result from \cite{silva1}.

\begin{lemma}
	\label{lem:technicalprojectionstuff}
	Let $c \in \Delta$ and let $\P \subseteq \Delta$ be an $i$-panel. Let $c'=\proj_\P(c)$ and $\dist(c,c') = n$. Let $d\in\ball_{n+1}(c)\setminus\P$, let $d'=\proj_\P(d)$, and assume that $d'\neq c'$, i.e., $d' \in \sphere_{n+1}(c)$.
	
	Then $d'$ is $j$-adjacent to some chamber in $\sphere_n(c)$ with $m_{ij}=2$.
\end{lemma}
\begin{proof}
    This is \cite[Lemma 2.15]{silva1}.
\end{proof}

Here is our generalization to panel-closed sets.
\begin{lemma}
	\label{lem:convextechnicalprojectionstuff}
	Let $C \subseteq \Delta$ be panel-closed and let $\P$ be an $i$-panel as in \itemref{lem:C-spheres}{b}, with $\P\cap\sphere_n(C) = \{ c' \}$.
	Let $d\in\ball_{n+1}(C)\setminus\P$, let $d'=\proj_\P(d)$, and assume that $d' \neq c'$, i.e., $d'\in\sphere_{n+1}(C)$.
	
	Then either $d$ is contained in an $i$-panel parallel to $\P$, or $d'$ is $j$-adjacent to some chamber in $\sphere_n(C)$ with $m_{ij}=2$.
\end{lemma}
\begin{proof}
    By \cref{prop:C-concave}, there is a minimal gallery $\gamma = (d', e, \dots, d)$ from $d'$ to $d$ that is concave with respect to $C$.
    By concavity, either $e \in \sphere_n(C)$, or all chambers of $\gamma$ lie in $\sphere_{n+1}(C)$.
    If $e \in \sphere_n(C)$, then the result already follows from \itemref{lem:C-closingsquares}{1} applied on $c' \sim_i d' \sim_j e$.
    
    Assume now that all chambers of $\gamma$ lie in $\sphere_{n+1}(C)$.
    We will show by induction on $\dist(d, \P)$ that $d$ is contained in an $i$-panel parallel to $\P$.
    Of course, this is obvious if $\dist(d, \P) = 0$, so assume $\dist(d, \P) \geq 1$, so that $e$ really exists. (If $\dist(d, \P)= 1$, we set $e=d$.)
    We can now apply \itemref{lem:C-closingsquares}{2} on $e \sim_j d' \sim_i c'$ to get $m_{ij} = 2$ and to find a chamber $e' \in \sphere_n(C)$ with $e' \sim_i e$.
    In particular, $\P$ is parallel to the $i$-panel $\P'$ through $e$ and $e'$, and $\dist(d, \P') = \dist(d, \P) - 1$.
    We have $\proj_{\P'}(d) = e$, and the subgallery $\gamma' = (e,\dots,d)$ of $\gamma$ is still a concave gallery completely contained in $\sphere_{n+1}(C)$.
    Hence we can apply the induction hypothesis to find that $d$ is contained in an $i$-panel parallel to $\P'$.
    Since parallelism is an equivalence relation, the result follows.
%   (OLD PROOF)
%	Write $c_0=\proj_C(c)$ and $d_0=\proj_C(d)$. If $c_0=d_0$, then we can immediately apply \cref{lem:technicalprojectionstuff}. Otherwise, let $d'\sim e\sim{\cdots}\sim d$ be a minimal gallery, concave with respect to $c_0$ (as in \cref{prop:concave}). If $\dist(c_0,e)\leq n+1$, we can again apply \cref{lem:technicalprojectionstuff}, with $e$ in place of $d$. In the other case where $\dist(c_0,e) = n+2$, concavity implies that $\dist(c_0,d) = \dist(c_0,d') + \dist(d',d)$. Together with the gate property (\cref{lem:convexgateproperty}) we now have two minimal galleries joining $c_0$ to $d$, one passing through $d'$ and one passing through $d_0$.
%	
%	The types of those galleries represent the same Coxeter group element, and are hence homotopic. In particular, by \cref{lem:homotopy}, we can find two chambers $e$ and $e'$ on the minimal gallery
%	\[ c_0 \sim \dots \sim e' \sim_i e \sim \dots \sim d \]
%	that passes through $d_0$, such that the panel containing $e\sim_i e'$ is parallel to $\P$ (containing $c\sim_i d'$). If $d=e$, then indeed $d$ is contained in a panel parallel to~$\P$. Otherwise $\{e,e'\}\subseteq\ball_n(C)$, and then all chambers $\{c,d',e,e'\}$ lie in a common residue $\R$ of type $i\cup i^\perp$. In this case, the intersection $\R\cap\sphere_n(C)$ contains some chamber $j$-adjacent to $d'$ with $j\in i^\perp$.
\end{proof}

Here is our promised ``extension result''.

\begin{proposition}
\index{convex}
	\label{prop:evenmoretechnicalextensionstuff}
	Let $C \subseteq \Delta$ be panel-closed. Let $\F$ be local data over $I$, and let $g\in\Aut(\Delta)$ be an automorphism mapping chambers in $C$ to harmonious chambers (with respect to $\F$).
	
	Then there exists an automorphism $h\in\Aut(\Delta)$ with the following properties:
	\begin{enumerate}
		\item $\restrict{g}{C}=\restrict{h}{C}$;
		\item for each $i$ and for each $i$-panel $\P$, either $\P$ is parallel to a panel contained in~$C$, or we have $\sigma_\lambda(h,\P)\in F_i$ (or both).
	\end{enumerate}
	Note that, in particular, $h$ maps all chambers in $\Delta$ to harmonious chambers.
\end{proposition}
\begin{proof}
	We will recursively construct a sequence of elements $g_n \in \Aut(\Delta)$ such that $\restrict{g_n}{C}=\restrict{h}{C}$ for all $n$, such that property (ii) holds for all panels $\P$ contained in $\ball_n(C)$, and such that $g_n$ and $g_m$ agree on the ball $\ball_m(C)$ whenever $m<n$. Note that $g_n$ maps chambers in $\ball_n(C)$ to harmonious chambers. For $n=0$, take $g_0=g$.

%	Next, we verify the following claim: \emph{if an $i$-panel $\P$ intersects both $\sphere_n(C)$ and $\sphere_{n+1}(C)$ nontrivially, where $n\geq 1$, then the intersection $\P\cap\sphere_n(C)$ is a single chamber}. Suppose $\{d_1,d_2\}\subseteq\P\cap\sphere_n(C)$ with $d_1\neq d_2$. If $\proj_C(d_1)\neq\proj_C(d_2)$, then those projections are $i$-adjacent, so that $\P$ is parallel to a panel contained in $C$, and $\P\subseteq\sphere_n(C)$. Hence it must be the case that $\proj_C(d_1)=\proj_C(d_2)$; let $c$ denote the common projection. Then the projection $\proj_\P(c)$ of $c$ onto $\P$ must be a chamber in $\sphere_{n-1}(C)$, and the intersection $\P\cap\sphere_{n+1}(C)$ is again empty. Our claim follows.

	Now assume that $n\geq 0$ and that we have constructed $g_n$ with all the required properties. In order to define $g_{n+1}$, we will construct an automorphism $h_n$ that stabilizes $\ball_n(g_n\acts C)$ pointwise and that fixes the mismatching local actions at $\sphere_{n+1}(g_n\acts C)$ --- we can then set $g_{n+1}=h_n\circ g_n$. Already let $C'=g_n\acts C=g\acts C$.
	
	For convenience, we will call a panel \emph{$n$-defective} if it intersects both $\sphere_n(C)$ and $\sphere_{n+1}(C)$  and it does not satisfy property (ii) with respect to $g_n$. Explicitly, if an $i$-panel $\P$ is $n$-defective, then the local action $\sigma_\lambda(g_n,\P)\notin F_i$ and $\P$ is not parallel to any panel contained in $C$. Observe that, if property~(ii) holds for some panel, then it holds for every parallel panel as well, since parallelism is transitive and local actions on parallel panels agree. This implies that an $n$-defective panel cannot be parallel to any panel in $\ball_n(C)$.
	
	Let $\P$ be an $n$-defective $i$-panel. By \cref{lem:C-spheres}, $\P$ intersects $\sphere_n(C)$ in a single chamber $c$. Let $\P'=g_n\acts\P$ and $c' = g_n\acts c$. By the induction hypothesis, $c$ and $c'$ are harmonious. Hence we can find a permutation $f_\P\in F_i$ such that $f_\P\acts\lambda_i(c)=\lambda_i(c')$, and a permutation $\pi_\P$ of the chambers in $\P'$ that makes the diagram below commute.
	\[\begin{tikzcd}[dims={5em}{4em}]
		\P \ar[d,"\lambda_i",swap] \ar[r,"g_n"] &
			\P' \ar[r,"\pi_\P"] &
			\P' \ar[d,"\lambda_i"]\\
		\Omega_i \ar[rr,"f_\P"] && \Omega_i
	\end{tikzcd}\]
%	By \cref{prop:extendpanel}, $\pi_\P$ extends to an automorphism $\widetilde\pi_\P$ that fixes all chambers whose projection onto $\P$ is fixed by $\pi_\P$. Note that $\pi_\P$ fixes $c'$ by construction, since both the local action of $g_n$ at~$\P$ and the target local action $f(\lambda_i(c),\lambda_i(c'))$ map the color $\lambda_i(c)$ to the same image $\lambda_i(c')$. Unlike in the proof of \cref{prop:colortransformation}, the automorphism $\widetilde\pi_\P$ does not necessarily fix $\ball_{n+1}(g\acts C)\setminus\P'$. Indeed, let $d\in\ball_{n+1}(g\acts C)\setminus\P'$ and consider $d'=\proj_{\P'}(d)$. If $\dist(d',C)=n$, or in other words if $c'=d'$, then $d'$ is fixed by $\pi_\P$ and $d$ is fixed by $\widetilde\pi_\P$. Otherwise $\dist(d',C)=n+1$. \Cref{lem:convextechnicalprojectionstuff} then yields that either the $i$-panel containing $d$ is parallel to $\P'$, or there exists some chamber $e'\in\sphere_n(C)$ such that $d'\sim_j e'$ with $m_{ij}=2$.
	By \cite[Proposition 4.2]{caprace2014}, $\pi_\P$ extends to an automorphism $\widetilde\pi_\P \in \Aut(\Delta)$ that fixes all chambers whose projection onto $\P$ is fixed by $\pi_\P$. Note that $\pi_\P$ fixes $c'$ by construction.
	Consider $d\in\ball_{n+1}(C')\setminus\P'$ and denote $d'=\proj_{\P'}(d)$. If $\dist(d',C')=n$, or in other words if $c'=d'$, then $d'$ is fixed by $\pi_\P$ and $d$ is fixed by $\widetilde\pi_\P$. Otherwise $\dist(d',C')=n+1$. By \cref{lem:convextechnicalprojectionstuff}, $\P'$ is parallel either to the $i$-panel containing $d$, or to an $i$-panel contained in $\ball_n(C')$. However, we have excluded the latter possibility by assuming that $\P$ violates property (ii).
	
	We have thus constructed, for every $n$-defective panel $\P$, an auto\-morphism $\widetilde\pi_\P$ of the building with the property that all chambers of $\ball_{n+1}(C')$ that are moved by $\widetilde\pi_\P$ are contained in $\sphere_{n+1}(C')\cap\R$, with $\R$ the set of all chambers in a panel parallel to $\P'$ (in other words, the residue of type $i\cup i^\perp$ containing $\P'$).

	Next, we claim that no chamber in $\sphere_{n+1}(C')$ lies in more than one $n$-defective panel. Indeed, suppose that $c_1\sim_i d\sim_j c_2$ with $d\in\sphere_{n+1}(C')$ and $c_1,c_2\in\ball_n(C')$. Then by \itemref{lem:C-closingsquares}{1}, the $i$-panel (and also the $j$-panel) through $d$ is parallel to a panel in $\ball_n(C')$, so that it cannot be $n$-defective.
	
	\medskip
	
	We now define
	\[h_n = \prod_\P \widetilde\pi_\P \in \Aut(\Delta),\]
	where the product runs over any set of arbitrarily chosen representatives of all equivalence classes of parallel $n$-defective panels, in an arbitrary order. Note that $h_n$ leaves invariant the set $\ball_n(C')$, since all factors do. In particular, the automorphism $g_{n+1}=h_n\circ g_n$ satisfies $\restrict{g_{n+1}}{C}=\restrict{g_n}{C}=\restrict{g}{C}$.
	
	Moreover, let $\P$ be any panel contained in $\ball_{n+1}(C)$. If $\P\subseteq\ball_n(C)$, then property (ii) holds by the induction hypothesis. If $\P\subseteq\ball_{n+1}(C)$, then $\P$ is parallel to some panel contained in $\ball_n(C)$ and again property (ii) holds by induction. It only remains to verify that property (ii) is valid for an $n$-defective panel $\P$ with respect to $g_{n+1}$. Let $\widetilde\P$ be the representative of $\P$ in the class of parallel $n$-defective panels. Then
	\[\sigma_\lambda(g_{n+1},\P)
		= \sigma_\lambda(g_{n+1},\widetilde\P)
		= \sigma_\lambda(h_n,g_n\acts\widetilde\P) \circ \sigma_\lambda(g_n,\widetilde\P)
		= \pi_{\widetilde\P} \circ \sigma_\lambda(g_n,\widetilde\P)
		= f_{\widetilde\P}\in F_i.\]
	Finally, by construction, $g_{n+1}$ agrees with $g_n$ on the ball $\ball_n(C)$. The sequence of automorphisms $g_0,g_1,g_2,\ldots$ thus obtained, converges to an automorphism satisfying the desired properties.
\end{proof}

\begin{remark}
    The intuitive reason for the assumption that $C$ be panel-closed in the above proposition is the following: If a partial automorphism is defined on a nontrivial subset of a panel only, we would need to extend it to the full panel first. This would require us to assume the partial local actions to be extendable to a full permutation in the local group in the first place. Panel-closure is a sufficient assumption to get rid of this additional technicality.
\end{remark}

%------------------------------------------------------------------------

\section{Restricted universal groups for right-angled buildings}
\label{sec:restrunivdef}

\begin{definition}%[restricted universal group (1)]
	\label{def:restricted1}
	Let $\F$ be a collection of permutation groups $F_i\leq\Sym(\Omega_i)$, indexed by $i\in I$. Then we define %the \emph{restricted universal group} of $\F$ over $\Delta$ is by definition the group
	\begin{multline*}
	   \G(\F) = \bigl\{g\in\Aut(\Delta) \bigm| \text{$\sigma_\lambda(g,\P)\in F_i$ for every $i\in I$ and} \\
	       \text{\emph{all but finitely many} $i$-panels $\P$}\bigr\}.
	\end{multline*}
	Note that by \cref{lem:localcomposition}, $\G(\F)$ is closed under composition and inversion, and is hence a subgroup of $\Aut(\Delta)$. We clearly have inclusions
    \[ \U(\F) \leq \G(\F) \leq \Aut(\Delta) . \]
\end{definition}

\begin{definition}
	Let $g\in \G(\F)$. An $i$-panel $\P$ with $\sigma_\lambda(g,\P)\notin F_i$ will be called a \emph{singularity} of $g$. We denote the (finite) set of all singularities of $g$ by $S(g)$.
\end{definition}

Even though the definition of $\G(\F)$ may suggest otherwise, the local action at a singularity cannot just be any permutation.
Indeed, some panels cannot be a singularity at all: Due to \cref{prop:parallellocals}, the set of panels parallel to any singularity must be finite.
This can go wrong in two different ways.
The first is easy:
\begin{lemma}
	\label{lem:locallyinfinitesingularities}
	Assume that $i,j\in I$ such that $m_{ij}=2$. If $q_j$ is infinite, then a panel of type $i$ can never be a singularity of an element $g\in\G(\F)$.
\end{lemma}
\begin{proof}
	Since the $i$-panels in a residue of type $\{i,j\}$ with $m_{ij}=2$ are parallel, this follows immediately from \cref{prop:parallellocals}.
\end{proof}

Even if all parameters $q_i$ are finite, there can be an infinite number of panels parallel to a given panel.
This situation is captured by the following definition.

\begin{definition}\label{def:ladder}
\begin{enumerate}
    \item 
    	A \emph{ladder} is a rank three building of type
    	\[M = \begin{pmatrix} 1 & 2 & 2\\ 2 & 1 & \mathclap{\infty}\\ 2 & \mathclap{\infty} & 1\end{pmatrix}
    	\qquad
    	\begin{tikzpicture}%[y=8mm]
    		\path (0,0) node[myvertex] (B) {} +(.5,.75) node[myvertex] (A) {} +(1,0) node[myvertex] (C) {};
    		\draw[myedge] (B) -- node[above] {$\infty$} (C);
    	\end{tikzpicture} \quad . \]
    	The \emph{rungs} of a ladder are the panels of the type corresponding to the isolated node of the diagram. We have visualized the ladder's Coxeter complex in \cref{fig:coxeterladder}.
    \item 
    	A \emph{ladder of $\Delta$} is a residue of $\Delta$ that has the type of a ladder, and its \emph{rungs} will be the corresponding panels of $\Delta$. Clearly, a ladder always has infinitely many rungs.
\end{enumerate}
\end{definition}

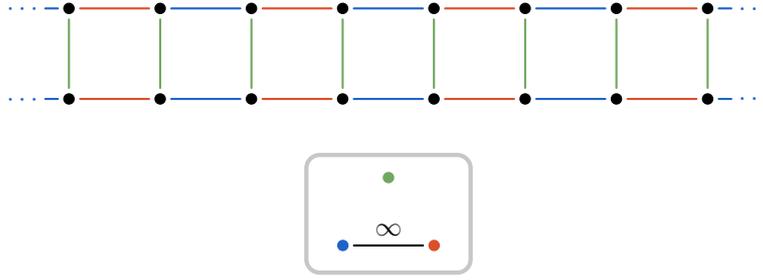
\begin{figure}[!ht]
	\centering
	\begin{tikzpicture}[x=12mm,y=12mm]
		%chambers
%		\fill[mylightgray] (-.48,-.48) rectangle (7.48,1.48);
% 		\draw[myedge,ultra thick,line cap=butt,mydarkgray] (-.48,.5) -- (7.48,.5);
%		\foreach\i in {0,...,6}
%	 		\draw[myedge,ultra thick,line cap=butt,mydarkgray] (\i+.5,-.48) -- (\i+.5,1.48);
		%chambers
		\foreach\i in {0,...,7}
			\foreach\j in {0,1}
				\node[myvertex] (P\i\j) at (\i,\j) {};
		%edges
		\foreach\i in {0,...,7}
			\draw[myedge,ugentgreen] (P\i0) -- (P\i1);
		\foreach\j in {0,1}
			{\draw[myedge,ugentred] (P0\j) -- (P1\j) (P2\j) -- (P3\j) (P4\j) -- (P5\j) (P6\j) -- (P7\j);
			\draw[myedge,ugentblue] (P1\j) -- (P2\j) (P3\j) -- (P4\j) (P5\j) -- (P6\j);
			\draw[myedge,ugentblue,-dots] (P0\j) -- (-.26,\j);
			\draw[myedge,ugentblue,-dots] (P7\j) -- (7.26,\j);
			}
	\end{tikzpicture}
	\par\vspace*{4ex}
	\begin{tikzpicture}%[y=8mm]
		\path (0,0) node[myvertex,ugentblue] (B) {}
			+(.5,.75) node[myvertex,ugentgreen] (A) {}
			+(1,0) node[myvertex,ugentred] (C) {};
		\draw[myedge] (B) -- node[above] {$\infty$} (C);
		\draw[myedge,ultra thick,mydarkgray,rounded corners=5pt] (-.4,-.3) rectangle (1.4,1);
	\end{tikzpicture}
%	\par\bigskip
%	\begin{tikzpicture}[y=8mm]
%		\path (0,0) node[myvertex,ugentblue] (A) {}
%			++(1,0) node[myvertex,ugentred] (B) {}
%			++(.75,0) node[myvertex,draw=ugentblue,fill=white] (C) {}
%			++(1,0) node[myvertex,draw=ugentred,fill=white] (D) {};
%		\draw[myedge] (A) -- node[above] {$\infty$} (B) (C) -- node[above] {$\infty$} (D);
%	\end{tikzpicture}
\caption{The Coxeter complex of a ladder.}
% Nodes \tikz\node[myvertex,fill=ugentblue,draw=ugentblue] {}; and \tikz\node[myvertex,fill=ugentred,draw=ugentred] {}; correspond to solid edges, node \tikz\node[myvertex,fill=white,draw=ugentblue] {}; to dotted edges.
\label{fig:coxeterladder}
\end{figure}

\begin{lemma}
	\label{lem:rungsnotsingular}
	Let $\R$ be a ladder in $\Delta$. Then a rung of\, $\R$ can never be a singularity of an element $g\in\G(\F)$.
\end{lemma}
\begin{proof}
	Since the rungs of a ladder are parallel, this follows immediately from \cref{prop:parallellocals}.
\end{proof}

In fact, the occurrence of an unbounded set of parallel panels implies the existence of ladders in $\Delta$:
\begin{proposition}
	\label{prop:unboundedrungs}
	Let $\P$ be an $i$-panel of $\Delta$ such that the set of all panels parallel to $\P$ is unbounded. Then $\P$ is the rung of a ladder in $\Delta$.
\end{proposition}
\begin{proof}
	Let $\P'$ be a panel such that $\dist(\P,\P') \geq |I|$. Let $c\in \P$ and let $c'=\proj_{\P'}(c)$. A minimal gallery $\gamma$ from $c$ to $c'$ has only types in $i^\perp$ (by \cref{prop:parallelinrab}) and has length $\dist(\P,\P') > |i^\perp|$. By the pigeonhole principle, there is a type $j\in i^\perp$ that occurs at least twice in the type of $\gamma$. In between, there must exist a type $k\in i^\perp$ in $\gamma$ such that $m_{jk}=\infty$, since the type of $\gamma$ is a reduced word. Hence, we have found three types $\{i,j,k\}$ satisfying $m_{ij}=m_{ik}=2$ and $m_{jk}=\infty$, i.e., the $i$-panel $\P$ is the rung of at least one ladder in $\Delta$.
\end{proof}

Because of \cref{lem:rungsnotsingular}, it is possible that allowing for a finite number of singularities in fact does not expand the group at all.
\begin{definition}
	Let $G$ be a simple, undirected graph. A vertex~$v$ of $G$ will be called a \emph{rung} if there exist two other, adjacent vertices $w_1\sim w_2$ in $G$ such that $v$ is adjacent to neither $w_1$ nor $w_2$. We call the graph $G$ \emph{ladderful} if every vertex is a rung.
	The right-angled diagram $M$ is called \emph{ladderful} if its underlying graph is ladderful.
	In \cref{fig:smallladderfulgraphs}, we have drawn all ladderful diagrams of rank $\leq 6$.
\end{definition}

%Computational results for graphs on at most eleven vertices are presented in \cref{fig:ladderdiagramstable}. In addition, as an illustration, we include all ten irreducible ladderful diagrams in \cref{fig:smallladderfulgraphs}. The percentage of ladderful graphs versus the total number of simple undirected graphs looks quite interesting, and we are interested in the asymptotic behaviour as $n$ goes to infinity. However, what mostly catches the eye are the numbers of ladderless diagrams --- these are precisely the partition numbers from number theory and combinatorics. There is a simple explanation.

%\begin{proposition}
%	A simple undirected graph is ladderless if and only if its complement is a disjoint union of complete graphs.
%\end{proposition}
%\begin{proof}
%	A graph is ladderless if and only if its complement graph has no path graph as an induced subgraph on three vertices. This is then easily shown to imply that every connected component of the complement graph must be a complete graph.
%\end{proof}

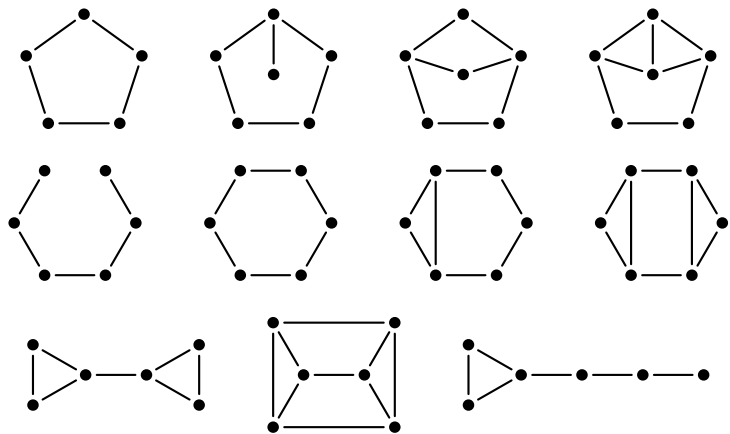
\begin{figure}[!th]
\newcommand{\pentagon}{\foreach\i in {0,...,4} \node[myvertex] (P\i) at (-72*\i-126:1) {}; \node[myvertex] (P5) at (0,0) {};}
\newcommand{\hexagon}{\foreach\i in {0,...,5} \node[myvertex] (P\i) at (-60*\i-120:1) {};}
	\centering
	\begin{tikzpicture}[x=8mm,y=8mm]
		\foreach\i in {0,...,4} \node[myvertex] (P\i) at (-72*\i-126:1) {};
		\draw[myedge] (P2) -- (P3) -- (P4) -- (P0) -- (P1) -- (P2);
	\end{tikzpicture}
	\qquad
	\begin{tikzpicture}[x=8mm,y=8mm]
		\pentagon;
		\draw[myedge] (P2) -- (P3) -- (P4) -- (P0) -- (P1) -- (P2) -- (P5);
	\end{tikzpicture}
	\qquad
	\begin{tikzpicture}[x=8mm,y=8mm]
		\pentagon;
		\draw[myedge] (P1) -- (P2) -- (P3) -- (P4) -- (P0) -- (P1) -- (P5) -- (P3);
	\end{tikzpicture}
	\qquad
	\begin{tikzpicture}[x=8mm,y=8mm]
		\pentagon;
		\draw[myedge] (P1) -- (P2) -- (P3) -- (P4) -- (P0) -- (P1) -- (P5) -- (P3) (P5) -- (P2);
	\end{tikzpicture}
	\par\bigskip\medskip
	\begin{tikzpicture}[x=8mm,y=8mm]
		\hexagon;
		\draw[myedge] (P3) -- (P4) -- (P5) -- (P0) -- (P1) -- (P2);
	\end{tikzpicture}
	\qquad
	\begin{tikzpicture}[x=8mm,y=8mm]
		\hexagon;
		\draw[myedge] (P0) -- (P1) -- (P2) -- (P3) -- (P4) -- (P5) -- (P0);
	\end{tikzpicture}
	\qquad
	\begin{tikzpicture}[x=8mm,y=8mm]
		\hexagon;
		\draw[myedge] (P0) -- (P1) -- (P2) -- (P3) -- (P4) -- (P5) -- (P0) -- (P2);
	\end{tikzpicture}
	\qquad
	\begin{tikzpicture}[x=8mm,y=8mm]
		\hexagon;
		\draw[myedge] (P0) -- (P1) -- (P2) -- (P3) -- (P4) -- (P5) -- (P0) -- (P2) (P3) -- (P5);
	\end{tikzpicture}
	\par\bigskip\medskip
	\begin{tikzpicture}[x=8mm,y=8mm]
		\path (0,0) node[myvertex] (P0) {} +(-150:1) node[myvertex] (P1) {} +(150:1) node[myvertex] (P2) {}
			++(1,0) node[myvertex] (P3) {} +(30:1) node[myvertex] (P4) {} +(-30:1) node[myvertex] (P5) {};
		\draw[myedge] (P0) -- (P1) -- (P2) -- (P0) -- (P3) -- (P4) -- (P5) -- (P3);
	\end{tikzpicture}
	\qquad
	\begin{tikzpicture}[x=8mm,y=8mm]
		\path (0,0) node[myvertex] (P0) {} +(-120:1) node[myvertex] (P1) {} +(120:1) node[myvertex] (P2) {}
			++(1,0) node[myvertex] (P3) {} +(60:1) node[myvertex] (P4) {} +(-60:1) node[myvertex] (P5) {};
		\draw[myedge] (P0) -- (P1) -- (P2) -- (P0) -- (P3) -- (P4) -- (P5) -- (P3) (P2) -- (P4) (P1) -- (P5);
	\end{tikzpicture}
	\qquad
	\begin{tikzpicture}[x=8mm,y=8mm]
		\path (3,0) node[myvertex] (P0) {} (2,0) node[myvertex] (P1) {}
			(1,0) node[myvertex] (P2) {} (0,0) node[myvertex] (P3) {}
			(150:1) node[myvertex] (P4) {} (-150:1) node[myvertex] (P5) {};
		\draw[myedge] (P0) -- (P1) -- (P2) -- (P3) -- (P4) -- (P5) -- (P3);
	\end{tikzpicture}
	\caption{All irreducible ladderful diagrams of rank $\leq 6$. (We have omitted the labels $\infty$). The first diagram is the only one of rank $5$; it is the diagram of the Bourdon building $I_{5,2}$.}
	\label{fig:smallladderfulgraphs}
\end{figure}

\begin{corollary}\label{cor:ladderful}
    If $M$ is ladderful, then $\G(\F) = \U(\F)$.
\end{corollary}
\begin{proof}
    Since $M$ is ladderful, every panel of $\Delta$ is the rung of some ladder in $\Delta$.
    By \cref{lem:rungsnotsingular}, this implies that an element $g \in \G(\F)$ cannot have any singularities at all.
    Hence $\G(\F) = \U(\F)$.
\end{proof}

Even for singularities that are not the rung of a ladder, there is some restriction on the local action. Recall the notion of Young overgroups from \cref{def:young}.

%\begin{definition}[Young subgroup]
%	We can associate a subgroup of $\Sym(\Omega)$ to every partition of $\Omega$, containing all permutations that stabilize the partition's blocks. A subgroup obtained in this fashion is called a \emph{Young subgroup} of $\Sym(\Omega)$ and is naturally isomorphic to the direct product of the symmetric groups on the blocks. In particular, for a permutation group $F\leq\Sym(\Omega)$, we call the Young subgroup associated to the partition of $\Omega$ into orbits of $F$ the \emph{Young overgroup of $F$} and denote it by $\widehat F$.
%	
%	Note that we always have $F \leq \widehat F \leq \Sym(\Omega)$. Moreover, $F = \widehat F$ if and only if $F$ is itself a Young subgroup, and $\widehat F = \Sym(\Omega)$ if and only if $F$ is transitive.
%\end{definition}

\begin{proposition}
	\label{prop:localinyoung}
	Let $g\in\G(\F)$ and let $\P$ be a singularity of $g$ of type $i$. Assume that $i$ is not an isolated node of the diagram. Then at least one of the following holds:
	\begin{itemize}
		\item the local action $\sigma_\lambda(g,\P)$ is contained in the Young overgroup $\widehat F_i$ of the local group $F_i$, or
		\item $F_i$ has at least two infinite orbits, while $F_j$ has finite degree for all $j\neq i$.
	\end{itemize}
\end{proposition}
\begin{proof}
	Assume that $\sigma = \sigma_\lambda(g,\P)$ is not contained in $\widehat F_i$, so there exists a color $x\in\Omega_i$ such that $x$ and $\sigma\!\acts x$ are contained in \emph{different} $F_i$-orbits. Let $c\in\P$ be such that $\lambda_i(c)=x$, let $j\neq i$, and let $d$ be any chamber $j$-adjacent to $c$. Since $\lambda_i(c)=\lambda_i(d)$ and $\lambda_i(g\acts c)=\lambda_i(g\acts d)$, the $i$-panel containing $d$ is again a singularity. As there are only finitely many singularities, this implies that the parameter $q_j$ is finite, for every $j\neq i$.
	
	Next, we claim that both orbits $F_i\acts x$ and $F_i\acts(\sigma\!\acts x)$ are infinite. Assume by means of contradiction that $X = F_i\acts x$ is finite. Since $\sigma$ is a permutation and $x\in X$ with $\sigma\!\acts x\notin X$, there exists a color $y\notin X$ with $\sigma\!\acts y\in X$. Let $c'\in\P$ be such that $\lambda_i(c')=y$ and note that $c\neq c'$. Since $i$ is not an isolated node in the diagram, there exists some $k\in I$ with $m_{ik}=\infty$; choose $d\sim_k c'$. Since $\lambda_i(d) = \lambda_i(c') = y$ and $\lambda_i(g \acts d) = \lambda_i(g \acts c') = \sigma \acts y$ are in different $F_i$-orbits, the local action at the $i$-panels containing $d$ is again a permutation not contained in~$\widehat F_i$. We can thus repeat this construction to obtain an infinite chain
	\[  c \sim_i c' \sim_k d \sim_i d' \sim_k \dotsm , \]
	every chamber of which lies in a singularity of type $i$. This contradiction shows that the orbit $F_i\acts x$ cannot be finite. Similarly, $F_i\acts(\sigma\!\acts x)$ cannot be finite.
\end{proof}

It follows from \cref{prop:localinyoung} that local actions at singularities are contained in the Young overgroups of the local groups, except for possibly one single type, depending on the finiteness of the parameters. Motivated by this observation and by \cref{lem:locallyinfinitesingularities}, it makes sense to restrict our setting to \emph{locally finite} right-angled buildings for the most interesting results. In particular, we then have the following corollary.

\begin{corollary}
	\label{cor:younggroupsinlocfincase}
	Assume that $\Delta$ is locally finite. Then we have an inclusion $\G(\F) \leq \U(\Fwidehat)$, with $\Fwidehat$ the local data obtained from the Young overgroups of the local groups in $\F$.
\end{corollary}
\begin{proof}
	This follows immediately from \cref{prop:localinyoung}.
\end{proof}

\myblock{From now on, we will always assume that the building $\Delta$ is locally finite. In particular, the local groups $F_i$ are permutation groups of finite degree.}

\begin{definition}
	\label{def:restricted2}
	Let $\F$ and $\Facute$ be two collections of local data for $\Delta$, satisfying
	\[F_i \leq \acute F_i \leq \widehat F_i \leq \Sym(\Omega_i)\]
	for every $i\in I$. In particular, $F_i$ and $\acute F_i$ have identical orbits.
	We define the \emph{restricted universal group} of $\F$ and $\Facute$ over $\Delta$ as
	\[\G(\F,\Facute) = \G(\F) \cap{} \U(\Facute).\]
	In other words, $\G(\F,\Facute)$ is the group of all automorphisms that locally act like permutations in $F_i$ but with a finite number of exceptions for which the local action still follows a prescribed behavior, namely a permutation in $\acute F_i$. We will continue to refer to those exceptions as \emph{singularities}, i.e., a singularity of $g \in \G(\F,\Facute)$ is an $i$-panel $\P$ such that $\sigma_\lambda(g,\P)\in\acute F_i\setminus F_i$. The set of all singularities will still be denoted as $S(g)$.
\end{definition}

The local groups $\acute F_i$ can be chosen between $F_i$ and $\widehat F_i$. In the extreme cases, we clearly have
\[ \G(\F,\F) = \U(\F)  \qquad\text{and}\qquad  \G(\F,\Fwidehat) = \G(\F) , \]
where the latter equality follows by \cref{cor:younggroupsinlocfincase}.

\begin{remark}
	Restricted universal groups were first introduced by Adrien Le Boudec in \cite{leboudec} in the setting of trees. The name originates from \cite{caprace_dense}, where the authors remark the analogy with \emph{restricted direct products}. There are a few related constructions in the literature, but especially in topological group theory, the restricted direct product of a collection of groups $\{G_i\}_{i\in I}$ with respect to subgroups $\{H_i\leq G_i\}_{i\in I}$ (over any index set $I$) is defined to be the subgroup
	\[ \Bigl\{ (g_i)_{i\in I}\in \prod_{i\in I} G_i \Bigm| \text{$g_i\in H_i$ for all but finitely many $i\in I$}\Bigr\}\]
	of the standard direct product. Just like the restricted direct product imposes additional constraints on all but finitely many factors, automorphisms in $\G(\F,\Facute)$ are essentially automorphisms in $\U(\Facute)$ with additional constraints on all but finitely many panels.
\end{remark}

%--------------------------------------------------------

\section{Permutational properties}\label{sec:permutational}

It is natural to focus on irreducible buildings. However, passing from reducible to irreducible buildings depends in a subtle way on the existence of isolated nodes (although it is not difficult).
\begin{lemma}\label{lem:reducible}
	Let $\Delta$ be a reducible right-angled building over $I$. Let $J_1,\ldots, J_m$ be the connected components of the diagram with at least two nodes, and let $J$ be their union. Let $k_1,\ldots,k_n$ be the isolated nodes of the diagram, and let $K$ be their union.
	\begin{enumerate}
		\item If $m=0$ (i.e., all nodes are isolated), then \[ \G_\Delta(\F,\Facute) = \U_\Delta(\Facute) \cong \acute F_{k_1} \times {\cdots} \times \acute F_{k_n} . \]
		\item If $m=1$, then
		      \begin{align*}
		      \G_\Delta(\F,\Facute)
				&\cong \G_{\R_J}(\restrict\F J,\,\restrict\Facute J) \times \U_{\R_K}(\restrict\F K) \\
				&\cong \G_{\R_J}(\restrict\F J,\,\restrict\Facute J) \times F_{k_1} \times {\cdots} \times F_{k_n} ,
		      \end{align*}
			where $\R_J$ and $\R_K$ are residues of types $J$ and $K$, respectively.
		\item If $m\geq 2$, then \[ \G_\Delta(\F,\Facute) = \U_\Delta(\F) \cong \U_{\R_1}\!\left(\textstyle\restrict\F{J_1}\right) \times {\cdots} \times \U_{\R_m}\!\left(\textstyle\restrict\F{J_m}\right) . \]
	\end{enumerate}
	In particular, if the diagram is reducible and has no isolated nodes, then $\G_\Delta(\F,\Facute) = \U_\Delta(\F)$.
\end{lemma}
\begin{proof}
	\begin{enumerate}
        \item If $m=0$, then $\Delta$ is a finite building so that all panels can be singularities at the same time, hence $\G_\Delta(\F,\Facute) = \U_\Delta(\Facute)$.
        \item If $m=1$, then $\Delta$ is isomorphic to the direct product of an infinite building $\R_J$ with a finite building $\R_K$. Then, viewing the chambers of $\Delta$ as pairs of chambers in $\R_J$ and $\R_K$, we can identify $\G_\Delta(\F,\Facute)$ with a subgroup $H$ of $\U_{\R_J}(\restrict\Facute J) \times \U_{\R_K}(\restrict\Facute K)$. Since $\R_K$ is finite, each potential singularity $\P$ in $\R_J$ corresponds to finitely many panels in $\R_J \times \R_K$, hence $\G_{\R_J}(\restrict\F J,\,\restrict\Facute J) \times 1$ is contained in $H$. It is obvious that also $1 \times \U_{\R_K}(\restrict\F K) \leq H$, so we already get
            \[ \G_{\R_J}(\restrict\F J,\,\restrict\Facute J) \times \U_{\R_K}(\restrict\F K) \leq H . \]
            Conversely, if $g \in H$, then no panel of $\R_K$ can be a singularity for $g$ in $\R_J \times \R_K$ because $\R_J$ is infinite (or because it is contained in a ladder of $\Delta$), so the restriction of $g$ to $\R_K$ is contained in $\U_{\R_K}(\restrict\F K)$.
            It is obvious that the restriction of $g$ to $\R_J$ is contained in $\G_{\R_J}(\restrict\F J,\,\restrict\Facute J)$, and the conclusion follows.
        \item If $m=2$, then the diagram $M$ is ladderful, so by \cref{lem:rungsnotsingular}, we have $\G_\Delta(\F) = \U_\Delta(\F)$. It follows that
        \[ \G_\Delta(\F,\Facute) = \G_\Delta(\F) \cap \U_\Delta(\Facute) = \U_\Delta(\F) \cap \U_\Delta(\Facute) = \U_\Delta(\F) . \qedhere \]
	\end{enumerate}
\end{proof}

The next few propositions follow for free from the inclusions $\U(\F)\leq\G(\F,\Facute)\leq\U(\Facute)$ and the corresponding properties of the universal groups.
\begin{proposition}
	\label{prop:restrictedorbits}
	Two residues lie in the same orbit of $\G(\F,\Facute)$ if and only if they are of the same type and harmonious. % In particular, two chambers $c$ and $c'$ lie in the same orbit of $\G_\Delta(\F,\Facute)$ if and only if their colors $\lambda_i(c)$ and $\lambda_i(c')$ lie in the same $F_i$-orbit for every $i\in I$.
\end{proposition}
\begin{proof}
	For each $i\in I$, the orbits of the local groups $F_i$ and $\acute F_i$ agree. By \cref{prop:universalorbits}, the same is true for the orbits of $\U(\F)$ and $\U(\Facute)$ on $\Delta$. Since $\U(\F)\leq\G(\F,\Facute)\leq\U(\Facute)$, the result now follows from the same \cref{prop:universalorbits}.
\end{proof}

The next lemma yields in particular a converse to \cref{lem:rungsnotsingular}: every panel that is \emph{not} the rung of a ladder, is a singularity of some element of $\G(\F,\Facute)$.
\begin{lemma}
	\label{lem:restrictedlocallemma}
	Let $f$ be a permutation in $\acute F_i$ and let $\P$ be an $i$-panel. Then there exists an automorphism $g\in\U(\Facute)$ with the following properties:
	\begin{enumerate}[label={\rm (\alph*)}]
		\item\label{lem:restrictedlocallemma:a} $g$ stabilizes $\P$;
		\item\label{lem:restrictedlocallemma:b} the local action is equal to $f$ at every panel parallel to $\P$;
		\item\label{lem:restrictedlocallemma:c} the local action is a permutation in $F_i$ at every other panel.
	\end{enumerate}
	If $\P$ is not the rung of a ladder in $\Delta$, then $g\in\G(\F,\Facute)$.
\end{lemma}
\begin{proof}
    The permutation $f \in \acute F_i$ corresponds to a permutation $\tilde f \in \U_\P(\acute F_i)$.
	By \cref{prop:ext-res}, $\tilde f$ extends to an automorphism $h \in \U_\Delta(\Facute)$.
	We can now apply \cref{prop:evenmoretechnicalextensionstuff} with $C = \P$ as the convex panel-closed set to find the required $g \in \Aut(\Delta)$ satisfying all three conditions.
	(Recall that by \cref{prop:parallellocals}, condition~\ref{lem:restrictedlocallemma:b} is automatically satisfied.)
	
	In particular, $g \in \U(\Facute)$.
	Moreover, if $\P$ is not the rung of a ladder in $\Delta$, then by \cref{prop:unboundedrungs}, the set of panels parallel to $\P$ is bounded and hence finite (because $\Delta$ is locally finite); it then follows that $g\in\G(\F,\Facute)$.
\end{proof}

%  --- cobounded & combinatorially dense ---
%
%\begin{proposition}
%	The action of $\G(\F,\Facute)$ on $\Delta$ is cobounded.
%\end{proposition}
%\begin{proof}
%	This follows immediately from the fact that $\U(\F)\leq\G(\F,\Facute)$ and \cref{cor:cobounded}.
%\end{proof}
%
%\begin{proposition}
%	\label{prop:restricteduniversalcombdense}
%	Let $\Delta$ be a right-angled building such that the diagram does not have isolated nodes. Then the action of $\G(\F,\Facute)$ on $\Delta$ is combinatorially dense.
%\end{proposition}
%\begin{proof}
%	This follows immediately from the fact that $\U(\F)\leq\G(\F,\Facute)$ and \cref{cor:universalminimal}.
%\end{proof}

\begin{definition}
	\label{def:KnP}
	Let $\P$ be a panel of $\Delta$.
	\begin{enumerate}
        \item
            We define the subgroup
        	\[K_\P = \bigl\{g\in\G(\F,\Facute) \bigm| \text{$g$ stabilizes $\P$ and every singularity of $g$ is parallel to $\P$} \bigr\}.\]
        	Notice that if $\P$ is a rung in $\Delta$, then $K_P$ coincides with the panel stabilizer $\U(\F)_{\{\P\}}$.
        \item
            For each integer $n\geq 0$, we define the subgroup
        	\[K_{n,\P} = \bigl\{g\in\G(\F,\Facute) \bigm| \text{$g$ stabilizes $\P$ and $\dist(\P,\P')\leq n$ for every singularity $\P'$} \bigr\}.\]
	\end{enumerate}
	Note that both are indeed subgroups by \cref{lem:localcomposition}.
	(We do not include the local data $\F$ and $\Facute$ in the notation and hope that this will not cause any confusion.)
\end{definition}
\begin{lemma}
	\label{lem:KPfromlocal}
	Let $\P$ be an $i$-panel that is not a rung in $\Delta$.
	\begin{enumerate}
	    \item\label{lem:KPfromlocal:1}
        	Let $f\in \acute F_i$ be a permutation. Then there exists an automorphism $g\in K_\P$ such that $\sigma_\lambda(g,\P)=f$.
        \item\label{lem:KPfromlocal:2}
        	$\subgroupindex{K_\P}{\U(\F)_{\{\P\}}} = \subgroupindex{\acute F_i}{F_i}$.
	\end{enumerate}
\end{lemma}
\begin{proof}
    \Cref{lem:KPfromlocal:1} is a special case of \cref{prop:evenmoretechnicalextensionstuff}.
    By \cref{prop:parallellocals}, \ref{lem:KPfromlocal:2} now follows from \ref{lem:KPfromlocal:1}.
\end{proof}

\begin{proposition}
	\label{prop:restrictedcompgen}
	Let $\H$ be a set of representatives of harmonious panels of $\Delta$ that are not rungs. Then
	\[\G(\F,\Facute) = \bigl\langle \U(\F), K_\P \mid \P\in\H\bigr\rangle.\]
\end{proposition}
\begin{proof}
	Let $g\in\G(\F,\Facute)$. We proceed by induction on the number $|S(g)|$ of singularities. Of course, if $S(g)=\emptyset$ then $g\in\U(\F)$ and there is nothing to show, so assume that $\lvert S(g) \rvert \geq 1$ and let $\P$ be a singularity of type $i$. By \cref{lem:rungsnotsingular}, $\P$ is not a rung in $\Delta$.
	
	By definition of $\H$, there is an $i$-panel $\P_0\in\H$ such that $\P$ and $\P_0$ are harmonious. By \cref{prop:universalorbits,prop:restrictedorbits}, there is thus an automorphism $h\in\U(\F)$ such that $h\acts\P=\P_0$. By~\cref{prop:ext-res} applied on the residue $\P$, we can assume that the local action of $h$ at $\P$ is the identity. Next, let $\sigma = \sigma_\lambda(g,\P)$. Using \cref{lem:KPfromlocal} we can find an automorphism $h_0\in K_{\P_0}$ satisfying $\sigma_\lambda(h_0,\P_0)=\sigma^{-1}$.
	Finally, let $h' := h^{-1}\cdot h_0\cdot h \in K_\P$.
	\[\begin{tikzcd}[dims={5em}{4em}]
		\P \ar[rrr,"h'"] \ar[dr,"h"] \ar[dd,"\lambda_i",swap]
			&&& \P \ar[dd,"\lambda_i"] \ar[r,"g"]
			& g\acts\P \ar[dd,"\lambda_i"]\\
		& \P_0 \ar[dl,"\lambda_i",swap] \ar[r,"h_0"]
			& \P_0 \ar[dr,"\lambda_i"] \ar[ru,"h^{-1}"]\\
		\Omega_i \ar[rrr,"\sigma^{-1}"]
			&&& \Omega_i \ar[r,"\sigma"]
			& \Omega_i
	\end{tikzcd}\]
	% Then singularities of $h'$ are parallel to $\P$.
	Consider the automorphism $g' = g\cdot h'$. By construction (see diagram), the local action of $g'$ at~$\P$ is the identity. In particular, $g'$ does not have singularities parallel to~$\P$.
	
	Now let $\P'$ be any other panel that is not parallel to $\P$; then $h'\acts\P'$ is not parallel to $\P$ either. Since
	\[\sigma_\lambda(g',\P')
		= \sigma_\lambda(g,h'\acts\P') \circ \sigma_\lambda(h',\P'),\]
	where the second factor is known to be a permutation in $F_i$, we find that $\P'$ is a singularity of $g'$ if and only if $h'\acts\P'$ is a singularity of $g$.
	
	In other words, the number of singularities \emph{not parallel to $\P$} is identical for $g$~and~$g'$; however, by the previous paragraph, $g$ has singularities parallel to $\P$ whereas $g'$ has none.
	This implies that $g'$ has strictly less singularities than $g$. The induction hypothesis finishes the proof.
\end{proof}

\begin{remark}\label{rem:Hfinite}
    The set $\H$ in \cref{prop:restrictedcompgen} is finite. Indeed, denote the number of orbits of each local group $F_i$ by $m_i=\lvert\Omega_i/F_i\rvert$. Then by \cref{prop:restrictedorbits}, the action of $\G(\F,\Facute)$ on the $i$-panels has $\prod_{j\neq i} m_j$ orbits. Summing over $i$,
    	\[ \lvert\H\rvert
    		= \sum_{\substack{i\in I \\ \textrm{not of rung-type}}} \prod_{j\neq i} m_j .
    		% = \sum_{i\in I}\frac{1}{m_i} \cdot \prod_{i\in I} m_i.
    	\]
\end{remark}

\section{Topological properties}\label{sec:topological}

The topology on $\G(\F,\Facute)$ is somewhat delicate, because this subgroup of $\Aut(\Delta)$ is almost never closed; see \cref{cor:restrictedboring} below. In particular, the induced topology on $\G(\F,\Facute)$ cannot be locally compact in this case (see, e.g., \cite[Chapter II, Theorem 5.11]{HR}).

In order to make $\G(\F,\Facute)$ into a tdlc group, we will have to ensure that the embedding $\U(\F)\hookrightarrow\G(\F,\Facute)$ be a continuous open map. In fact, there will be a unique way to do this. An elegant way to formulate this is provided by the following lemma from \cite{neretingroups}.
\begin{lemma}\label{lem:open-cont topology}
    Let $G$ be an abstract group containing a topological group $H$ as a subgroup.
    Then $G$ admits a unique group topology such that the inclusion $H \hookrightarrow G$ is continuous and open, provided that
    \begin{equation}\label{eq:open}
        gUg' \cap H \text{ is open in } H \tag{$*$}
    \end{equation}
    for all open subsets $U \subseteq H$ and for all $g,g' \in G$.
    
    Moreover, if $G$ is itself contained in a larger topological group $A$ in which $H$ is a closed subgroup (with the subgroup topology inherited from $A$), then this condition~\eqref{eq:open} is automatically satisfied.
\end{lemma}
\begin{proof}
    For the first statement, see \cite[Lemma 8.4]{neretingroups}.
    The second statement is now clear: if $U \subseteq H$ is open in $H$, then it is also open in $A$ and hence also $gUg'$ is open in $A$, so that $gUg' \cap H$ is indeed open in $H$.
\end{proof}

%Let $\mathfrak B$ be the collection of all compact open subgroups of $\U(\F)\leq\G(\F)$. By van Dantzig's theorem (\cref{cor:dantzig}), $\mathfrak B$ is an identity neighborhood basis for the topology on $\U(\F)$. Note that $\mathfrak B$ is a filter base on $\G(\F)$ that satisfies all properties in \cref{lem:technicalfiltergrouptopology}. Indeed, (i) and (ii) are immediate, and in order to check (iii), take any $U\in\mathfrak B$ and $g\in\G(\F)$. Then by \cref{cor:GcommensuratesUstuff}, the intersection of $U$ and $g\cdot U\cdot g^{-1}$ has finite index in $U$ and is hence again compact and open in $\U(\F)$. We now obtain a unique, well-defined group topology on $\G(\F)$ by \cref{lem:technicalfiltergrouptopology}.

\begin{definition}
	\label{def:leboudectopology}
	For each choice of local data $\F$ and $\Facute$, we endow $\G(\F)$ and $\G(\F,\Facute)$ with the group topology obtained from \cref{lem:open-cont topology} using the embeddings $\U(\F) \hookrightarrow \G(\F) \leq \Aut(\Delta)$ and $\U(\F) \hookrightarrow \G(\F,\Facute) \leq \Aut(\Delta)$.
\end{definition}

\begin{remark}
    \begin{enumerate}
        \item 
            Since $\G(\F,\Facute) = \G(\F) \cap \U(\Facute)$, the topologies on $\G(\F)$ and $\G(\F,\Facute)$ are compatible, in the sense that $\G(\F,\Facute)$ is a closed subgroup of $\G(\F)$ (with the subspace topology).
        \item 
            Note that when $\F=\Facute$, the topology on $\G(\F,\Facute)=\U(\F)$ agrees with the permutation topology, but in general it does not. Indeed, in our topology on $\G(\F,\Facute)$, a subgroup $H\leq\G(\F,\Facute)$ is open if and only if $H$ contains a pointwise stabilizer of a finite set of chambers \emph{intersected with $\U(\F)$}. In other words, the topology on $\G(\F,\Facute)$ in \cref{def:leboudectopology} is finer than the permutation topology.
    \end{enumerate}
\end{remark}

\begin{proposition}
\label{prop:restrictedtdlc}
	Endow $\G(\F,\Facute)$ with the topology from \cref{def:leboudectopology}.
	Then $\G(\F,\Facute)$ is a totally disconnected locally compact (tdlc) group.
\end{proposition}
\begin{proof}
	This follows immediately from the fact that $\U(\F)$ is an open tdlc subgroup of~$\G(\F,\Facute)$ for this topology.
\end{proof}

Using our previous knowledge of $\U(\F)$, we have the following immediate corollary.
\begin{corollary}
	\label{prop:restricteddiscrete}
	$\G(\F,\Facute)$ is discrete if and only if every local group in $\F$ is free.
\end{corollary}
\begin{proof}
	By definition of the topology, $\G(\F,\Facute)$ is discrete if and only if $\U(\F)$ is. The result is now an immediate consequence of \cite[Proposition 3.14]{bossaert}.
\end{proof}

In the following proposition, we can assume without loss of generality that the local data $\F$ and $\Facute$ already take into account the rung restriction (\cref{lem:rungsnotsingular}) by demanding that $F_i = \acute F_i$ for every $i\in I$ that is the type of a rung.

\begin{proposition}
	\label{prop:restrictedclosure}
	Let $\F$ and $\Facute$ be the local data over $I$ such that $F_i = \acute F_i$ for every $i\in I$ that is the type of a rung in $\Delta$.
	Then the closure of $\G(\F,\Facute)$ in $\Aut(\Delta)$ is $\U(\Facute)$.
\end{proposition}
\begin{proof}
	Let $g\in\U(\Facute)$. It suffices to find a sequence of automorphisms $g_n\in\G(\F,\Facute)$, where $n\in\mathbb N$, converging to $g$ in the permutation topology on $\Aut(\Delta)$. Let $c\in\Delta$ be any chamber and, for every $n\in\mathbb N$, define the set $B_n$ to be the convex closure of the ball $\ball_n(c)$. Note that $B_n$ is panel-closed.
	
	\Cref{prop:evenmoretechnicalextensionstuff} yields an automorphism $g_n$ that agrees on $B_n$ with $g$. Moreover, the only panels where the local action of $g_n$ is not guaranteed to be a permutation in $\F$ are the panels parallel to some panel in $B_n$. By our assumption on the rungs, however, the parallel classes of singularities are bounded (\cref{prop:unboundedrungs}), and hence finite. Hence we have $g_n\in\G(\F,\Facute)$, and this finishes our proof.
\end{proof}

\begin{corollary}
	\label{cor:restrictedboring}
	Let $\F$ and $\Facute$ be the local data over $I$ such that $F_i = \acute F_i$ for every $i\in I$ that is the type of a rung in $\Delta$.
	Then the following are equivalent:
	\begin{enumerate}
		\item $\F=\Facute$; %$F_i=\acute F_i$ for every $i\in I$;
		\item $\G(\F,\Facute)=\U(\F)$;
		\item $\G(\F,\Facute)$ is closed in $\Aut(\Delta)$;
		\item chamber stabilizers of $\G(\F,\Facute)$ are compact;
		\item chamber stabilizers of $\G(\F,\Facute)$ are closed in $\Aut(\Delta)$.
	\end{enumerate}
\end{corollary}
\begin{proof}
	Implications (i) $\Rightarrow$ (ii) $\Rightarrow$ (iii) $\Rightarrow$ (iv) $\Rightarrow$ (v) are straightforward. For (iii) $\Rightarrow$ (i), if $\G(\F,\Facute)$ is closed, then $\G(\F,\Facute)=\U(\Facute)$ by \cref{prop:restrictedclosure}, and hence $\F=\Facute$. Finally, (v) $\Rightarrow$ (iii) is a general fact in the permutation topology: if a point stabilizer of a permutation group is closed, then the group itself is closed.
\end{proof}

\begin{proposition}
	\label{lem:KPcompact}
	For each panel $\P$ of $\Delta$ and each integer $n \geq 0$, the subgroups $K_\P$ and $K_{n,\P}$ of $\G(\F,\Facute)$ are compact open.
\end{proposition}
\begin{proof}
	Since $K_\P$ contains the setwise panel stabilizer $\U_{\{\P\}}$ as a subgroup of finite index, it suffices to note that $\U_{\{\P\}}$ is compact and open in $\U(\F)$ (as it is a union of finitely many cosets of chamber stabilizers), and hence in $\G(\F,\Facute)$ as well. The same argument works for $K_{n,\P}$ once we note that the set $\{c\in\Delta\mid\dist(c,\P)\leq n\}$ contains only finitely many panels.
\end{proof}

An immediate corollary is the following fact.
\begin{corollary}
	\label{cor:restrictedcompgen}
	The group $\G(\F,\Facute)$ is compactly generated.
\end{corollary}
\begin{proof}
    First, recall from \cite[Theorem 4.8]{bossaert} that $\U(\F)$ is compactly generated.
	Combining \cref{lem:KPcompact,rem:Hfinite,prop:restrictedcompgen}, we now see that $\G(\F,\Facute)$ is generated by a finite number of compactly generated subgroups.
\end{proof}

From \cref{lem:KPcompact}, we can also deduce that panel stabilizers, although not compact, are still quite close to being compact, in the following precise sense.
\begin{definition}[regionally compact]
	A topological group is said to be \emph{regionally compact} if it is the increasing union of a family of compact open subgroups.
\end{definition}
\begin{remark}
    In the literature, regionally compact groups are also called \emph{locally elliptic}. The authors of \cite[Remark 1.0.1]{caprace_dense} argue why this terminology is best avoided: ``locally'' is suggestive of a property to be satisfied in some identity neighborhood basis, which is not the case here.
\end{remark}
	
\begin{corollary}
	\label{prop:panelstabregionallycompact}
	Let $\P$ be a panel of $\Delta$. Then the stabilizer $\G(\F,\Facute)_{\{\P\}}$ is regionally compact.
\end{corollary}
\begin{proof}
	Since we can write
	\[ \G(\F,\Facute)_{\{\P\}} = \dircup_{n\in\mathbb N} K_{n,\P} \,, \]
	the result follows immediately from \cref{lem:KPcompact}.
\end{proof}

%--------------------------------------------------------
%--------------------------------------------------------

\section{Simplicity}\label{sec:simplicity}

In order to study simplicity of the restricted universal groups $\G(\F,\Facute)$, we first observe that by \cref{lem:reducible}, this group will never be simple if $\Delta$ is a reducible building. We will nevertheless state some of our auxiliary results without assuming $\Delta$ to be irreducible.

Let us first recall the known simplicity result for the universal groups $\U(\F)$ from \cite{bossaert}.
As in \emph{loc.\@~cit.}, we will write
\[ \U(\F)^+ := \langle \U(\F)_c \mid c \in \Delta \rangle \]
for the open normal subgroup of $\U(\F)$ generated by all chamber stabilizers.
\begin{definition}
    A \emph{vertex cover} of a graph $\Gamma = (V, E)$ is a subset $W \subseteq V$ such that each edge $e \in E$ has at least one of its two vertices in $W$. A \emph{vertex cover} of the diagram $M$ is a vertex cover of its underlying graph (omitting the labels $\infty$).
\end{definition}
\begin{theorem}\label{thm:Usimple}
    Let $\Delta$ be irreducible.
    Assume that not all local groups $F_i$ are free (or equivalently, that $\U(\F)$ is non-discrete).
    Then the universal group $\U(\F)$ is (abstractly) simply if and only if the local groups $F_i$ are generated by point stabilizers for all $i \in I$ and are transitive for all $i$ in some vertex cover of the diagram $M$.
\end{theorem}
\begin{proof}
    This is \cite[Theorem 5.7]{bossaert}.
    Notice that the condition that not all local groups $F_i$ be free was erroneously omitted in \emph{loc.\@~cit.}, but is required to exclude the possibility that $\U(\F)^+ = 1$, in which case we have not been able to determine whether or not the discrete group $\U(\F)$ is simple.
\end{proof}

The proof of this fact relies in an essential way on the fact that $\U(\F)$ is a \emph{closed} subgroup of $\Aut(\Delta)$.
For the restricted universal groups, we will only be able to show that $\G(\F,\Facute)$ is \emph{virtually} simple (i.e., that it contains a simple subgroup of finite index) under the same conditions for $\U(\Facute)$ to be simple. The proof is inspired by the proof of the related result for restricted universal groups for trees by Pierre-Emmanuel Caprace, Colin Reid, and Phillip Wesolek in \cite[Proposition 9.2.13]{caprace_dense}, but we will first need some preparation specific to the setting of right-angled buildings.

\medskip

We begin by recalling some definitions related to \emph{wings} and \emph{tree-walls} in right-angled buildings.
\begin{definition}
    \begin{enumerate}
        \item 
        	Let $i\in I$. An equivalence class of parallel $i$-panels is called an \emph{$i$-tree-wall} of $\Delta$. By slight abuse of notation, we often identify an $i$-tree-wall $\T$ with the set of chambers $\{c\in\P\mid\P\in\T\}$.
        \item\label{def:treewalltree}
        	Let $V_1$ be the set of all $i$-tree-walls of $\Delta$, and let $V_2$ be the set of all residues of $\Delta$ of type $I\setminus\{i\}$. Define a bipartite graph $\Gamma_i$ with vertex set $V_1\sqcup V_2$, where an $i$-tree-wall $\T$ is adjacent to a residue $\R$ of type $I\setminus\{i\}$ if and only if $\T\cap\R\neq\emptyset$. We call this graph the \emph{$i$-tree-wall tree} of $\Delta$.
        	By \cite[Proposition 2.39]{silva1}, each tree-wall tree is indeed a tree.
        	Notice that the \emph{edges} of $\Gamma_i$ correspond to residues of $\Delta$ of type $i^\perp$.
    \end{enumerate}
\end{definition}

\begin{definition}
    \begin{enumerate}
        \item 
        	Let $J\subseteq I$ and let $c\in\Delta$ be a chamber. Then the set of chambers
    		\[X_J(c) = \{d\in\Delta \mid \proj_\R(d) = c\}\]
    		where $\R$ is the $J$-residue containing $c$, is called the \emph{$J$-wing of $c$}. If $J=\{j\}$ is a singleton, we simply write $X_j(c)$ and call it the \emph{$j$-wing of $c$}.
        \item 
        	For each chamber $c\in\Delta$ and each $i\in I$, we define the subgroup
        	\[
        		V_i(c) = \{g\in \Aut(\Delta) \mid \text{$g\acts d=d$ for all chambers $d\notin X_i(c)$}\} \leq \Aut(\Delta).
        	\]
        	Note that $V_i(c)$ fixes all chambers of the $i$-tree-wall containing $c$.
        	Moreover, for each subgroup $G \leq \Aut(\Delta)$, we will write
        	\[
        		V_i^G(c) = V_i(c) \cap G.
        	\]
    \end{enumerate}
\end{definition}

Before we proceed, we focus on properties of the action of arbitrary subgroups of $\Aut(\Delta)$.
\begin{definition}
	\label{def:minimalbuilding}
	Let $G \leq \Aut(\Delta)$.
	\begin{enumerate}
		\item The action is \emph{cobounded} if there exists a constant $r\in\mathbb N$ and a chamber $c\in\Delta$ such that
			\[\Delta = \bigcup_{g\in G}\ball_r(g\acts c),\]
			i.e.,~if every chamber lies at a uniformly bounded distance to some $G$-orbit.
		\item The action is \emph{minimal} if $G$ leaves no nontrivial convex subset of $\Delta$ invariant.
%		\item The action is \emph{geometrically dense} if it is minimal and if moreover the induced action on the Davis realisation $\mathbb K(\Delta)$ has no fixed point at infinity in the boundary $\partial\mathbb K(\Delta)$.
		\item The action is \emph{combinatorially dense} if it is minimal and moreover, for every $i\in I$, the induced action on the $i$-tree-wall tree $\Gamma_i$ has no fixed point at infinity.
	\end{enumerate}
\end{definition}
These three notions are strongly related to each other; see \cref{lem:coboundeddense,prop:minimalcharacterisations}.
\begin{definition}
	For each $i\in I$, define the \emph{$i$-distance} $\dist_i(c_1,c_2)$ between any two chambers $c_1$ and $c_2$ as the number of occurences of type $i$ in a minimal gallery joining $c_1$ and $c_2$. Note that this number does not depend on the chosen minimal gallery. This distance function is only a pseudometric, since chambers in a common residue of type $I\setminus\{i\}$ are at $i$-distance zero.
\end{definition}
The $i$-distance function is quasi-isometric to the graph-theoretical edge distance in the $i$-tree-wall tree: this is the content of the following lemma.
\begin{lemma}
	\label{lem:quasiisometricaldistances}
	Let $c_1$ and $c_2$ be any two chambers of $\Delta$. Let $\R_1$ and $\R_2$ be the residues of type $i^\perp$ containing $c_1$ and $c_2$ respectively, viewed as edges in the $i$-tree-wall tree $\Gamma_i$. Then
	\[\dist_{\Gamma_i}(\R_1,\R_2) = 2\cdot\dist_i(c_1,c_2)+\epsilon \leq \dist(c_1,c_2),\qquad\text{where }\epsilon\in\{-1,0,1\}.\]
\end{lemma}
\begin{proof}
	The inequality is clear. For the equality, pick any minimal gallery $\gamma$ from $c_1$ to $c_2$ in $\Delta$ and consider the induced path in the $i$-tree-wall tree, which is a path without backtracking. Notice that we have a one-to-one correspondence between $i$-adjacencies on $\gamma$ on the one hand, and pairs of adjacent edges in $\Gamma_i$ sharing a common vertex in the bipartition class of residues of type $i\cup i^\perp$ on the other hand. Consequently, if the path in $\Gamma_i$ has even length, say $2n$, there are $n$ adjacencies of type $i$ on $\gamma$, and if the path has odd length, say $2n+1$, there are either $n$ or $n+1$ adjacencies of type $i$ on $\gamma$. Our conclusion follows.
\end{proof}
\begin{corollary}\label{cor:cobounded-TWT}
	Let $G \leq \Aut(\Delta)$.
	If the action of $G$ is cobounded, then its induced action on each tree-wall tree $\Gamma_i$ is also cobounded.
\end{corollary}
\begin{proof}
    This is an immediate consequence of \cref{lem:quasiisometricaldistances}.
\end{proof}

\begin{proposition}
	\label{lem:coboundeddense}
		Let $G \leq \Aut(\Delta)$ and assume that the diagram $M$ has no isolated nodes. If the action of $G$ is cobounded, then it is combinatorially dense.
\end{proposition}
\begin{proof}
    The fact that the action of $G$ is \emph{minimal} was already shown in \cite[Lemma 4.5]{bossaert}.
	Now let $i\in I$ and consider the $i$-tree-wall tree $\Gamma_i$. By \cref{cor:cobounded-TWT}, the induced action of $G$ on $\Gamma_i$ is also cobounded. Note that $\Gamma_i$ is unbounded, as $i$ is not an isolated node of the diagram. If $G$ would fix a point at infinity $\xi$ of the $i$-tree-wall tree, then it would stabilize all \emph{horospheres} through $\xi$. However, for each horosphere, there are vertices of the tree at unbounded distance to that horosphere, contradicting the coboundedness.
\end{proof}
\begin{corollary}\label{cor:combinatorially dense}
    Assume that the diagram $M$ has no isolated nodes.
    Then the universal groups $\U(\F)$ and the restricted universal groups $\G(\F,\Facute)$ are combinatorially dense.
\end{corollary}
\begin{proof}
    It suffices to observe that already $\U(\mathbf{1})$ is cobounded (where $\mathbf{1}$ is the local data with $F_i = 1$ for all $i$), as in \cite[Corollary 3.4]{bossaert}, so also any larger subgroup of $\Aut(\Delta)$ is cobounded. The result now follows from \cref{lem:coboundeddense}.
\end{proof}

\begin{proposition}
	\label{prop:minimalcharacterisations}
	Let $G \leq \Aut(\Delta)$ and assume that the diagram $M$ has no isolated nodes. Then the following are equivalent:
	\begin{enumerate}[label={\rm (\alph*)},leftmargin=6ex]
		\item the action of $G$ on $\Delta$ is minimal;
		\item for each $i\in I$, the induced action of $G$ on the $i$-tree-wall tree $\Gamma_i$ is minimal.
	\end{enumerate}
	As a corollary, the following are equivalent:
	\begin{enumerate}[label={\rm (\alph*)},leftmargin=6ex]
		\item the action of\, $G$ on $\Delta$ is combinatorially dense;
		\item for each $i\in I$, the induced action of $G$ on the $i$-tree-wall tree $\Gamma_i$ is combinatorially dense.
	\end{enumerate}
\end{proposition}
\begin{proof}
	First assume $G$ does not act minimally on $\Delta$, i.e., there exists a nontrivial convex $G$-invariant subset $C$ of chambers. Let $c\sim_i d$ be such that $c\in C$ and $d\notin C$. Using absence of isolated nodes in the diagram, take $j\in I$ such that $m_{ij}=2$, and let $e\sim_j d$. By convexity, $e\notin C$. Now define $C'$ as the set of vertices in $\Gamma_i$ that correspond to residues of $\Delta$ having nonempty intersection with $C$,
	\[C' = \bigl\{\R\in\Res_{I\setminus\{i\}}(\Delta) \bigm| \R\cap C \neq \emptyset\bigr\}
		\cup \bigl\{\R\in\Res_{\{i\}\cup\{i\}^\perp}(\Delta) \bigm| \R\cap C \neq \emptyset\bigr\}.\]
	Note that $C'$ is a nonempty $G$-invariant subtree of $\Gamma_i$. Moreover, we claim that $e$ is not contained in any residue in $C'$. Indeed, if it were the case that $e\in\R\in C'$, let $e'\in \R\cap C$. The projection $\proj_\R(c)$ is either $d$ or $e$, depending on the type of~$\R$. By the gate property, there is a minimal gallery joining $c$ to $e'$ passing through either $d$ or $e$, contradicting the convexity of $C$. Hence, $C'$ is a \emph{nontrivial} $G$-invariant subtree, and the action of $G$ on $\Gamma_i$ is not minimal.
	
	Conversely, assume that $G$ does not act minimally on $\Gamma_i$ for some $i\in I$, i.e.~there exists a nontrivial $G$-invariant subtree $D$. Define $D'$ as the set of chambers of $\Delta$ that lie in a residue corresponding to a vertex in $D$,
	\[D' = \bigl\{c\in\Delta \bigm| \text{$c\in\R\in D$ for some residue $\R$ of type $I\setminus\{i\}$ or type $\{i\}\cup \{i\}^\perp$}\bigr\}.\]
	Clearly $D'$ is a nontrivial $G$-invariant subset of chambers. Let $c,c'\in D'$ and consider a minimal gallery $\gamma$ from $c$ to $d$ in $\Delta$. Then $\gamma$ descends to a path in~$\Gamma_i$ between the residues in $D$ containing $c$ and $d$. This path is contained in the subtree~$D$, hence $\gamma$ is contained in $D'$. In other words, $D'$ is a nontrivial \emph{convex} $G$-invariant subsystem, and the action of $G$ on $\Delta$ is not minimal.
\end{proof}

\begin{lemma}\label{prop:normaltreedense}
    Let $T$ be a tree and $G \leq \Aut(T)$.
    If $G$ is combinatorially dense, then so is any non-trivial normal subgroup $N \unlhd G$.
\end{lemma}
\begin{proof}
    This is \cite[Lemme 4.4]{tits_arbres}.
\end{proof}

For the next proposition, we will have to make the stronger assumption that $\Delta$ be \emph{irreducible}. (The statement is clearly false when $\Delta$ is reducible.)
\begin{proposition}
	\label{lem:combdensenormal}
	Assume that the diagram $M$ is irreducible. Let $G \leq \Aut(\Delta)$ be combinatorially dense and let $N\unlhd G$ be a nontrivial normal subgroup. Then also the action of $N$ on $\Delta$ is combinatorially dense.
\end{proposition}
\begin{proof}
    The subtle but important point here is that because $\Delta$ is irreducible, the induced action of $\Aut(\Delta)$ on each tree-wall tree $\Gamma_i$ is \emph{faithful}; see \cite[Lemma~3.19]{silva1}.
    In particular, the action of $N$ on each $\Gamma_i$ is non-trivial.
	The result now follows from \cref{prop:minimalcharacterisations,prop:normaltreedense}.
\end{proof}

Recall from \cite[Proposition 3.2]{tits_arbres} that if $T$ is a tree, then the automorphisms of~$T$ come in three flavors.
\begin{definition}
    Let $T$ be a tree and let $g \in \Aut(T)$.
    \begin{enumerate}
        \item We call $g$ an \emph{inversion} if $g$ inverts some edge of $T$.
        \item We call $g$ \emph{elliptic} if $g$ fixes a vertex.
        \item We call $g$ \emph{hyperbolic} otherwise, in which case there is a bi-infinite path $\gamma$ in $T$ left invariant by $g$ and on which $g$ induces a non-trivial translation. We call $\gamma$ the \emph{axis} of $g$.
    \end{enumerate}
    In particular, if $G \leq \Aut(T)$ acts \emph{without inversion} (i.e., $G$ fixes the bipartition of the vertex set of $T$), then each element of $G$ is either elliptic or hyperbolic.
\end{definition}
It is useful to transfer these notions to automorphisms of $\Delta$ by the action on their tree-wall trees.
Notice that the tree-wall trees have two genuinely different types of vertices, so inversions cannot occur.
\begin{definition}
    Let $g \in \Aut(\Delta)$ and let $i \in I$.
    \begin{enumerate}
        \item We call $g$ \emph{$i$-elliptic} if the induced action of $g$ on $\Gamma_i$ is elliptic.
        \item We call $g$ \emph{$i$-hyperbolic} otherwise. The axis $\gamma$ for the action of $g$ on $\Gamma_i$ will be called the \emph{$i$-axis} of $g$.
    \end{enumerate}
\end{definition}
We will need the following observation.
\begin{lemma}\label{prop:magictreestuff}
    Let $T$ be a tree and let $G \leq \Aut(T)$ act without inversion. Assume that the action is combinatorially dense.
    Let $e_1,e_2$ be two edges of $T$.
    Then there exists a hyperbolic element $g \in G$, the axis $\gamma$ of which contains both $e_1$ and $e_2$.
\end{lemma}
\begin{proof}
	We may of course assume that $e_1\neq e_2$. Let $T_1$ be the half-tree determined by $e_1$ not containing $e_2$ and let $T_2$ be the half-tree determined by $e_2$ not containing $e_1$. By \cite[Lemma 4.3]{leboudec}, there exist two hyperbolic elements $g_1$ and $g_2$ in $G$ with axes $\gamma_1 \subset T_1$ and $\gamma_2 \subset T_2$, respectively. In particular, the axes are disjoint, and the shortest path joining $\gamma_1$ and $\gamma_2$ passes through both $e_1$ and $e_2$. The product $g_1 g_2$ is then the required hyperbolic automorphism; its axis will contain this shortest path.
\end{proof}

\begin{proposition}
\label{hyperbolic}
	\label{lem:normalsubcontainshyperbolics}
	Assume that the diagram $M$ is irreducible. Let $i\in I$ and let $e_1,e_2$ be adjacent edges of the $i$-tree-wall tree $\Gamma_i$. Let $N\unlhd G$ be a nontrivial normal subgroup and assume that the action of $G$ on $\Delta$ is combinatorially dense. Then $N$ contains an $i$-hyperbolic automorphism, the axis of which contains both $e_1$ and $e_2$.
\end{proposition}
\begin{proof}
	By \cref{lem:combdensenormal}, the action of $N$ on $\Delta$ is combinatorially dense.
	By \cref{prop:minimalcharacterisations}, the induced action of $N$ on $\Gamma_i$ is combinatorially dense, so we can apply \cref{prop:magictreestuff}.
\end{proof}

Our next goal is to show that the groups $\G(\F,\Facute)$ satisfy the \emph{independence property}.
\begin{definition}
	Let $\P$ be an $i$-panel and let $\T$ be the corresponding $i$-tree-wall.
	Let $G \leq \Aut(\Delta)$ and let $G_{(\T)}$ be the pointwise stabilizer of the set of all chambers in $\T$.
	For each $c\in\P$, we define the morphism
	\[\varphi_{\P,c}\colon G_{(\T)}\to V_i(c),
		\quad\text{with }
		\varphi_{\P,c}(g)\colon\Delta\to\Delta\colon d\mapsto\begin{cases}
			g\acts d & \text{if $d\in X_i(c)$;}\\
			d & \text{otherwise.}
		\end{cases}\]
	The proof of \cite[Proposition 5.2]{caprace2014} shows that $\varphi_{\P,c}(g)$ is indeed an automorphism of $\Delta$, which is clearly contained in $V_i(c)$, but not necessarily in $V_i^G(c)$.
	We say that $G$ satisfies the \emph{independence property} when for each panel $\P$, each chamber $c \in \P$ and each $g \in G$, the automorphism $\varphi_{\P,c}(g)$ is again contained in $G$ (and hence in $V_i^G(c)$).
	If this is the case, then for each panel $\P$, the elements of $G$ can be decomposed into ``independent'' elements acting on each of the wings:
	\[
	   \varphi_\P \colon G_{(\T)} \xrightarrow{\sim} \prod_{c\,\in\,\P}  V_i^G(c) \colon g \mapsto \prod_{c\,\in\,\P} \varphi_{\P,c}(g) .
	\]
\end{definition}

\begin{proposition}
	\label{prop:restricteduniversalindependence}
	$\G(\F,\Facute)$ satisfies the independence property.
\end{proposition}
\begin{proof}
	Let $G = \G(\F,\Facute)$, let $\P$ be an $i$-panel, let $\T$ be the corresponding $i$-tree-wall and let $c \in \P$.
	Let $g \in G_{(T)}$ and $h = \varphi_{\P,c}(g) \in \Aut(\Delta)$.
	To show that $h \in G$, we will follow the proof of the independence property for $\U(\F)$ from \cite[Proposition 3.16]{silva1}, with as only new ingredient the fact that we have to allow a finite number of singularities.
	More precisely, if some $j$-panel $\P'$ is contained in $\T$ or in any of the other $i$-wings $X_i(d)$ with $d \neq c$, then $\P'$ is fixed by $h$; in particular, it cannot be a singularity.
	On the other hand, if $\P'$ is contained in the $i$-wing $X_i(c)$, then the local action of $h$ on $\P'$ coincides with the local action of $g$ on $\P'$, which, by assumption, is contained in $F_j$ up to a finite number of exceptions, for which it is contained in $\acute F_j$.
	This shows that indeed $h \in \G(\F,\Facute)$.
\end{proof}

We will now prepare for \cref{prop:normaltreewall}, which will be a crucial ingredient for our simplicity result.
The following lemma is a straightforward generalization of \cite[Lemma 3.17]{silva1}, which holds in a more general setting than only for universal groups.

\begin{lemma}
	\label{lem:auxnormaltreewall}
	Let $G \leq \Aut(\Delta)$.
	Let $i\in I$ and let $c$ and $d$ be two chambers in a common $i$-panel $\P$ of $\Delta$. Let $g\in G$ and assume that $\P$ and $g\acts\P$ are not parallel, that $\proj_\P(g\acts c)=d$ and that $\proj_{g\acts\P}(d)=g\acts c$. Consider an automorphism
	\[b\in\prod_{e\,\in\,\P\setminus\{c,d\}} V_i^G(e).\]
	Then there exists an automorphism $h\in G$ such that $b=[h,g]$.
\end{lemma}
\begin{proof}
    The proof of \cite[Lemma 3.17]{silva1} can be copied \emph{verbatim} to this more general setting.
\end{proof}

The proof of the following result is similar to that of \cite[Proposition~3.21]{silva1}, with the important difference that the proof in \emph{loc.\@~cit.} relies on the assumption that the local groups are transitive, which we do not assume. This is why we will use the stronger starting point from \cref{lem:normalsubcontainshyperbolics}.

\begin{proposition}
	\label{prop:normaltreewall}
	Let $\Delta$ be irreducible and let $G \leq \Aut(\Delta)$. Assume that $G$ is combinatorially dense. Let $N\unlhd G$ be a nontrivial normal subgroup. Then $N$ contains $V_i^G(c)$ for all $i \in I$ and all chambers $c \in \Delta$.
	
	In particular, if $G$ also satisfies the independence property, then $N$ contains the fixator $G_{(\T)}$ of every $i$-tree-wall $\T$\!.
\end{proposition}
\begin{proof}
	Let $\T$ be an arbitrary $i$-tree-wall and let $\P$ be an $i$-panel of $\T$\!. In the $i$-tree-wall tree $\Gamma_i$, $\T$ corresponds to a vertex $v$ and the edges incident to $v$ correspond to the chambers in $\P$. Let $c,d$ be two distinct chambers of $\P$. By \cref{lem:normalsubcontainshyperbolics}, there exists some $i$-hyperbolic element $g\in N$ such that the axis of $g$ in $\Gamma_i$ contains the two edges corresponding to $c$ and $d$. We can assume that $d$ is the chamber pointing towards the attracting end of $g$ in $\partial\Gamma_i$ and $c$ towards the repelling end, replacing $g$ by $g^{-1}$ if necessary.
	
	Note that $g\acts\T\neq\T$ since $g$ is an $i$-hyperbolic automorphism. Moreover, by construction, we have $\proj_\P(g\acts c)=d$ and $\proj_{g\acts\P}(d)=g\acts c$. Hence we can apply \cref{lem:auxnormaltreewall} to obtain that
	\begin{equation}\label{eq:Vi}
	   \prod_{e\,\in\,\P\setminus\{c,d\}} V_i^G(e) \subseteq N. \tag{$\ast$}
	\end{equation}
	Since $\Delta$ is assumed to be thick, there is at least one such chamber $e\,\in\,\P\setminus\{c,d\}$. Now recall that $c$ and $d$ were chosen arbitrarily in $\P$, so we can repeat the proof with different pairs $\{c,d\}$ and multiply the resulting equations \eqref{eq:Vi} to obtain that 
	\[ \prod_{e\,\in\,\P} V_i^G(e) \subseteq N. \]
	Moreover, if $G$ satisfies the independence property, then $G_{(\T)} = \prod_{e\,\in\,\P} V_i^G(e)$, showing the last statement.
\end{proof}

In particular, we can now deduce that $\G(\F,\Facute)$ is either discrete or \emph{monolithic}, i.e., it contains a unique minimal normal subgroup (called the \emph{monolith} of the group). Moreover, the monolith will turn out to be simple.
Again, we have been inspired by \cite[\S 4]{leboudec}.
\begin{proposition}
	\label{prop:restricteduniversalmonolithic}
	Let $\Delta$ be irreducible and let $\F,\Facute$ be local data such that not all local groups $F_i$ are free. Then $\G(\F,\Facute)$ is monolithic; the monolith is the subgroup $N$ generated by all tree-wall fixators.
	Moreover, $N$ is simple.
\end{proposition}
\begin{proof}
	Write $\G = \G(\F,\Facute)$ and $N = \langle \G_{(\T)} \rangle$, where $\T$ ranges over all tree-walls. Note that $N$ is nontrivial, since not all local groups act freely. Thanks to \cref{cor:combinatorially dense,prop:restricteduniversalindependence}, we may apply \cref{prop:normaltreewall} to obtain that every normal subgroup of~$\G$ contains $N$. Hence $\G$ is monolithic with monolith $N$.

	To show that $N$ is simple, let $A\trianglelefteq N$ be the intersection of all nontrivial normal subgroups of $N$. This is a characteristic subgroup of $N$, which is therefore normal in~$\G$.
	Notice that $N$ is combinatorially dense by \cref{lem:combdensenormal}, so by \cref{prop:normaltreewall} (applied with $N$ in place of $G$), any nontrivial normal subgroup of $N$ contains $V_i^N(c)$, for all $i \in I$ and all $c \in \Delta$. It follows that also $A$ contains all $V_i^N(c)$.
	Now recall that $V_i^{\G}(c) \leq \G_{(\T)} \leq N$ (where $\T$ is the $i$-tree-wall through $c$), so $V_i^N(c) = V_i^{\G}(c) \neq 1$.
	
	It follows that $A$ is a nontrivial normal subgroup of $\G$.
	By \cref{prop:normaltreewall} again, $A$ contains the monolith $N$, so that in fact $A=N$, showing that $N$ is indeed simple.
\end{proof}

As an interesting corollary of \cref{prop:restricteduniversalmonolithic}, we then obtain the following.
\begin{proposition}
	Assume that the diagram $M$ is irreducible and not ladderful.
	Then every nontrivial normal subgroup of $\G(\F,\Facute)$ is open.
	In particular, $\G(\F,\Facute)$ is topologically simple if and only if it is abstractly simple.
\end{proposition}
\begin{proof}
	Write $\G = \G(\F,\Facute)$.
	If every local group $F_i$ acts freely, then $\G$ is discrete and there is nothing to prove.
	
	Otherwise, $\G$ is monolithic by \cref{prop:restricteduniversalmonolithic} and the monolith is generated by tree-wall fixators. By the assumption on the diagram, there exists at least one index $i\in I$ such that an $i$-tree-wall~$\T$ contains only finitely many chambers. Then the fixator $\G_{(\T)}$ --- as the intersection of finitely many chamber stabilizers --- is open. Consequently, the monolith of $\G$ is open, and so is every non\-trivial normal subgroup of $\G$.
\end{proof}

We need one more result from \cite{bossaert}.
\begin{proposition}
	\label{prop:genbychamstab}
	Assume that $M$ has no isolated nodes. The following are equivalent.
	\begin{enumerate}[label={\rm (\alph*)}]
		\item\label{prop:genbychamstab:a} $\U(\F)^+=\U(\F)$;
		\item\label{prop:genbychamstab:b} $\U(\F)^+$ has finite index in $\U(\F)$;
		\item\label{prop:genbychamstab:c} the local groups are generated by point stabilisers for every $i\in I$, and are transitive for every $i$ in some vertex cover of (the underlying graph of) the diagram of $\Delta$.
	\end{enumerate}
\end{proposition}
\begin{proof}
    This is \cite[Theorem 5.6]{bossaert}, except that statement \ref{prop:genbychamstab:b} is not mentioned in the statement of \emph{loc.\@~cit.} but only in its proof, more precisely in the first paragraph on p.\@~884.
\end{proof}

We now come to our final simplicity result.
We again assume without loss of generality that $F_i = \acute F_i$ for every rung type $i\in I$.

\begin{theorem}
	\label{thm:restrictedvirtsimple}
	Let $\Delta$ be a thick irreducible right-angled building over an index set $I$ with $\lvert I \rvert \geq 2$. Let $\F$ and $\Facute$ be the local data as in \cref{def:restricted2}. Assume that $F_i = \acute F_i$ for every $i\in I$ that is the type of a rung. Moreover, assume that not all local groups $\acute F_i$ are free.%TODO NODIG?
	
	Then the restricted universal group $\G(\F,\Facute)$ is virtually simple if and only if $\acute F_i$ is generated by point stabilizers for every $i\in I$ and transitive for every $i$ in some vertex cover of the diagram of $\Delta$.
\end{theorem}
\begin{proof}
	Abbreviate $\G = \G(\F,\Facute)$ and $\U = \U(\Facute)$.
	First, suppose that $\G$ has a simple subgroup $H$ of finite index. Since $\G$ is infinite, $H$ is a normal subgroup of $\G$.
	On the other hand, by \cref{prop:restricteduniversalmonolithic}, $\G$ is monolithic, with monolith $N = \langle \G_{(\T)} \rangle$, where $\T$ ranges over all tree-walls. Since $N \leq H$ and $H$ is simple, we have $N = H$, so the monolith $N$ of $\G$ has finite index in $\G$.

	Now let $A = \U^+ \cap \G \unlhd \G$.
	By \cref{prop:restrictedclosure}, $\G$ is dense in $\U$. Since $\U^+$ is an open subgroup of $\U$, it follows that $A \neq 1$. Hence $A$ has finite index in~$\G$.
	 Taking the closure, it then follows that $\overline{A}$ has finite index in $\overline{\G} = \U$. Since $\overline{A} \leq \overline{\U^+} \cap \overline{\G} = \U^+$, we see that $\U^+$ has finite index in $\U$, and the characterisation follows from \cref{prop:genbychamstab}.

    \medskip
    
	Conversely, suppose that the local data $\Facute$ satisfies the assumptions postulated. Then by \cref{thm:Usimple}, $\U$ is simple. Let $N = \langle \G_{(\T)} \rangle$ be the open normal subgroup of $\G$ generated by all fixators of tree-walls $\T$ of $\Delta$. By \cref{prop:restricteduniversalmonolithic}, $N$ is a simple group and it is the monolith of $\G$. We will show that $N$ has finite index in~$\G$.
	
	The closure of $N \leq \Aut(\Delta)$ is a normal subgroup of $\U$. Hence $N$ is dense in~$\U$, and consequently in $\G$ as well. This implies that the $N$-orbits and $\G$-orbits on the building $\Delta$ agree: for any $c\in\Delta$ and $g\in\G$, the stabilizer $\G_c$ is open, hence the intersection $N \cap {g\cdot\G_c}$ is nonempty, and any automorphism $h$ in the intersection satisfies $h\acts c= g\acts c$. It follows that $\G = {\G_c}\cdot N$.
	
	Let $\P$ be a panel of $\Delta$. By the previous paragraph, $\G = {\G_{\{\P\}}}\cdot N$. Together with \cref{prop:panelstabregionallycompact}, we obtain an increasing union
	\[\G = \dircup_{n\in\mathbb N} H_n \cdot N,\]
	where the $H_n \leq \G_{\{\P\}}$ are compact open subgroups. On the other hand, $\G$ is compactly generated by \cref{cor:restrictedcompgen}. The family $\{H_n \cdot N\}_{n\in\mathbb N}$ defines an open cover of any compact generating set and hence $\G = H_n\cdot N$ for some $n\in\mathbb N$.
	
	In conclusion, $N$ is both open and cocompact in $\G$. The coset space ${\G}/N$ being both compact and discrete, it follows that indeed $N$ has finite index in $\G$, which concludes our proof.
\end{proof}

%------------------------------------------------------------------------

\bigskip
%\clearpage

\nocite{*}
\footnotesize
\bibliographystyle{alpha}
\bibliography{sources}

\begin{thebibliography}{DMdSS18}

\bibitem[AB08]{abramenkobrown}
Peter Abramenko and Kenneth~S. Brown.
\newblock {\em Buildings}, volume 248 of {\em Graduate Texts in Mathematics}.
\newblock Springer, New York, 2008.
\newblock Theory and applications.

\bibitem[BDM21]{bossaert}
Jens Bossaert and Tom De~Medts.
\newblock Topological and algebraic properties of universal groups for
  right-angled buildings.
\newblock {\em Forum Mathematicum}, 33(4):867--888, 2021.

\bibitem[BDM22]{cityproducts}
Jens Bossaert and Tom De~Medts.
\newblock City products of right-angled buildings and their universal groups.
\newblock {\em preprint}, 2022.

\bibitem[BM00]{burgermozes2000}
Marc Burger and Shahar Mozes.
\newblock Groups acting on trees: from local to global structure.
\newblock {\em Publications Math\'ematiques de l'IH\'ES}, 92:113--150, 2000.

\bibitem[Cap14]{caprace2014}
Pierre-Emmanuel Caprace.
\newblock Automorphism groups of right-angled buildings: simplicity and local
  splittings.
\newblock {\em Fundamenta Mathematicae}, 224(1):17--51, 2014.

\bibitem[CDM11]{CDM11}
Pierre-Emmanuel Caprace and Tom De~Medts.
\newblock Simple locally compact groups acting on trees and their germs of
  automorphisms.
\newblock {\em Transform. Groups}, 16(2):375--411, 2011.

\bibitem[CRW19]{caprace_dense}
Pierre-Emmanuel Caprace, Colin Reid, and Phillip Wesolek.
\newblock Approximating simple locally compact groups by their dense locally
  compact subgroups.
\newblock {\em International Mathematics Research Notices}, 2021(7):5037--5110,
  01 2019.

\bibitem[DMdS19]{silva2}
Tom De~Medts and Ana Filipa~Costa da~Silva.
\newblock Open subgroups of the automorphism group of a right-angled building.
\newblock {\em Geometriae Dedicata}, 203:1--23, 2019.

\bibitem[DMdSS18]{silva1}
Tom De~Medts, Ana Filipa~Costa da~Silva, and Koen Struyve.
\newblock Universal groups for right-angled buildings.
\newblock {\em Groups, Geometry and Dynamics}, 12(1):231--287, 2018.

\bibitem[GGT18]{universal-survey}
Alejandra Garrido, Yair Glasner, and Stephan Tornier.
\newblock Automorphism groups of trees: generalities and prescribed local
  actions.
\newblock In {\em New directions in locally compact groups}, volume 447 of {\em
  London Math. Soc. Lecture Note Ser.}, pages 92--116. Cambridge Univ. Press,
  Cambridge, 2018.

\bibitem[GL18]{neretingroups}
Lukasz Garncarek and Nir Lazarovich.
\newblock {\em The {N}eretin groups}, pages 131--144.
\newblock London Mathematical Society Lecture Note Series. Cambridge University
  Press, 2018.

\bibitem[HP03]{haglundpaulin}
Fr{\'e}d{\'e}ric Haglund and Fr{\'e}d{\'e}ric Paulin.
\newblock Constructions arborescentes d'immeubles.
\newblock {\em Mathematische Annalen}, 325(1):137--164, Jan 2003.

\bibitem[HR79]{HR}
Edwin Hewitt and Kenneth~A. Ross.
\newblock {\em Abstract harmonic analysis. {V}ol. {I}}, volume 115 of {\em
  Grundlehren der Mathematischen Wissenschaften [Fundamental Principles of
  Mathematical Sciences]}.
\newblock Springer-Verlag, Berlin-New York, second edition, 1979.
\newblock Structure of topological groups, integration theory, group
  representations.

\bibitem[LB16]{leboudec}
Adrien Le~Boudec.
\newblock Groups acting on trees with almost prescribed local action.
\newblock {\em Commentarii Mathematici Helvetici}, 91(2):253--293, 2016.

\bibitem[Tit70]{tits_arbres}
Jacques Tits.
\newblock {\em Sur le groupe des automorphismes d'un arbre}, pages 188--211.
\newblock Springer Berlin Heidelberg, 1970.

\bibitem[vD36]{vandantzig}
David van Dantzig.
\newblock Zur topologischen {A}lgebra. {III}. {B}rouwersche und {C}antorsche
  {G}ruppen.
\newblock {\em Compositio Mathematica}, 3:408--426, 1936.

\end{thebibliography}

\bigskip

\end{document}